\documentclass{amsart}
\usepackage{multicol, amsmath, amssymb, amsthm,epsf,graphicx,graphics}
\usepackage{color}
\usepackage{graphicx}
\usepackage{verbatim}
\usepackage{pstricks}
\usepackage{cite}
\newtheorem{thm}{Theorem}[section]
\newtheorem{cor}[thm]{Corollary}
\newtheorem{lem}[thm]{Lemma}
\newtheorem{prop}[thm]{Proposition}
\theoremstyle{definition}
\newtheorem{defn}[thm]{Definition}
\theoremstyle{remark}
\newtheorem{rem}[thm]{Remark}
\numberwithin{equation}{section}

\newcommand{\set}[1]{\left\{#1\right\}}
\newcommand{\Real}{\mathbb R}

\newcommand{\func}[1]{\ensuremath{\mathrm{#1} \:} }
\newcommand{\Hess}[0]{\func{Hess}}
\newcommand{\Div}[0]{\func{div}}
\newcommand{\pRect}[4]{S_{{#1}, {#2}}^{{#3}, {#4}}}

\newcommand{\re}[0]{\func{Re}}
\newcommand{\im}[0]{\func{Im}}

\title{Conformal Structure of Minimal Surfaces with Finite Topology}

\author{Jacob Bernstein and Christine Breiner}
\address{Dept. of Math, Massachusetts Institute of
Technology, Cambridge, MA  02139, USA}
\email{jbern@math.mit.edu}

\address{Dept. of Math,
Johns Hopkins University, Baltimore, MD 21218, USA}
\email{cbreiner@math.jhu.edu}
\thanks{The first author was supported in part by the NSF grant DMS-0606629}

\begin{document}

\begin{abstract}
In this paper we show that a complete, embedded minimal surface in
$\Real^3$, with finite topology and one end, is conformal to a
once-punctured compact Riemann surface. Moreover, using this
conformal structure and the embeddedness of the surface, we
examine the Weierstrass data and conclude that every such surface
has Weierstrass data asymptotic to that of the helicoid.  More
precisely, if $g$ is the stereographic projection of the Gauss
map, then in a neighborhood of the puncture, $g(p) = \exp(i\alpha
z(p) + F(p))$, where $\alpha \in \Real$, $z=x_3+ix_3^*$ is a
holomorphic coordinate defined in this neighborhood and $F(p)$ is
holomorphic in the neighborhood and extends over the puncture with
a zero there.  As a consequence, the end is asymptotic to a
helicoid.  This completes the understanding of the conformal and
geometric structure of the ends of complete, embedded minimal
surfaces in $\Real^3$ with finite topology.
\end{abstract}
\maketitle
\section{Introduction}
We apply the techniques of \cite{BB1} to study complete, embedded
minimal surfaces in $\Real^3$ with finite topology and one end. We
refer to the space of such surfaces as $\mathcal{E}(1)$. Notice
that we do not {a priori} assume the surfaces are properly
embedded.  This is because Colding and Minicozzi have shown, in Corollary 0.13 of
\cite{CY}, that every complete, embedded minimal surface with
finite topology is, in fact, properly embedded.  This fact will be
used implicitly throughout.  Surfaces in $\mathcal{E}(1)$ that
have genus zero have been completely classified by Meeks and
Rosenberg in \cite{MR} and consist of planes
and helicoids; thus we restrict attention to the subset
$\mathcal{E}(1,+)\subset \mathcal{E}(1)$ of surfaces that have
positive genus.  This space is non-trivial; the embedded genus one
helicoid, $\mathcal{H}$, constructed in \cite{WHW} by Hoffman,
Weber, and Wolf provides an example which, moreover, has the
property of being asymptotically helicoidal (see also
\cite{WHWPNAS} for a nice overview of their construction).

\begin{figure}
 \includegraphics[width=3in]{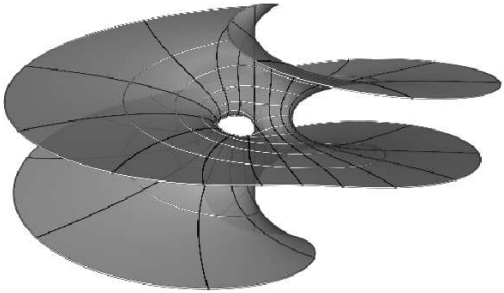}
\caption{A genus-one helicoid (Courtesy of Matthias
Weber)}\label{g1helicoid}
\end{figure}

The construction and study of $\mathcal{H}$ has a rich history.
Using the Weierstrass representation, Hoffman, Karcher, and Wei in
\cite{HKW} first constructed an immersed genus-one helicoid. See
Figure \ref{g1helicoid} for an image of this surface. Computer
graphics suggested it was embedded, but a rigorous construction of
an embedded genus-one helicoid followed only after Hoffman and Wei
proposed a new construction in \cite{HWe}. They considered the limit
of a family of screw-motion invariant minimal surfaces with periodic
handles and a helicoidal end. Weber, Hoffman, and Wolf confirmed the
existence of such a family of surfaces in \cite{WHW2} and ultimately
proved their embeddedness in \cite{WHW}, giving $\mathcal{H}$.
Hoffman, Weber, and Wolf conjecture that $\mathcal{H}$ is not only
the same surface as the one produced in \cite{HKW}, but is actually
the unique element in the class of ``symmetric'' genus-one
helicoids, that is, surfaces in $\mathcal{E}(1)$ with genus one and
containing two coordinate axes. Recently, Hoffman and White, in
\cite{HW}, used variational methods to give a different construction
of a symmetric, embedded genus-one helicoid; whether their
construction is $\mathcal{H}$ is unknown.

Building on \cite{HW}, in \cite{HW2}, Hoffman and White prove rigidity results for
properly immersed minimal surfaces with genus one and one end that, moreover, admit the
same symmetries as $\mathcal{H}$. In particular,
they show such surfaces are conformal to a punctured
torus and are asymptotic to a helicoid. In this paper, we prove that
any $\Sigma\in \mathcal{E}(1,+)$ is conformal to a once punctured,
compact Riemann surface and its Weierstrass data -- see \eqref{WeierstrassRep} -- has
helicoid-like behavior at the puncture:

\begin{thm} \label{ThirdMainThm} Any $\Sigma \in \mathcal{E}(1)$ is conformally a
punctured, compact Riemann surface.  Moreover, if the surface is not flat, then, after a rotation of $\Real^3$, the height
differential, $dh$, extends meromorphically over the puncture with
a double pole, as does the meromorphic one form $\frac{dg}{g}$.
\end{thm}

The proof of this result draws heavily on the fundamental work of
Colding and Minicozzi on the geometric structure of embedded
minimal surfaces in $\Real^3$ \cite{CM1,CM2,CM3,CM4,CM5}. Assuming
only mild conditions on the boundaries, they give a description of
the geometric structure of essentially all embedded minimal
surfaces with finite genus. From this structure, they deduce a
number of important consequences. These include: the one-sided
curvature estimate for embedded minimal disks \cite{CM4} -- an
effective version of the strong half-space theorem of Hoffman and
Meeks \cite{HoffmanMeeksHalfSpace}; a compactness result for a
sequence of embedded minimal disks in \cite{CM4} that requires no
a priori bounds on the curvature or area; and the settling of the
Calabi-Yau conjecture for embedded minimal surfaces of finite
topology \cite{CY}, i.e., a complete, embedded minimal
surface in $\Real^3$ of finite topology is properly embedded.

Colding and Minicozzi's work is also an essential ingredient in
understanding minimal surfaces with infinite total curvature,
i.e., complete surfaces with one end. Prior to their work, the
study of these surfaces required very strong assumptions on the
conformal structure and behavior of the Gauss map at the end.  For
examples we refer to Hoffman and McCuan \cite{HoffmanMcCuan} and
Hauswirth, Perez and Romon \cite{HPR}.  The latter authors
consider $E\subset \Real^3$, a complete embedded minimal annulus
with one compact boundary component and one end with infinite
total curvature. They assume, in addition, that $E$ is conformal
to a punctured disk, the Weierstrass data $(g,dh)$ has the
property that $dg/g$ and $dh$ extend across the puncture, and the
flux over the boundary of $E$ has zero vertical component.
Assuming all of this, they then deduce more precise information
about the asymptotic geometry of $E$.  Notice in a suitable
neighborhood of the puncture of a non-flat element of
$\mathcal{E}(1)$, Theorem \ref{ThirdMainThm} immediately implies
that these conditions are satisfied.

By using Colding and Minicozzi's work, in particular the
compactness result of \cite{CM4}, Meeks and Rosenberg were able to
remove such strong assumptions for disks.  Indeed, in \cite{MR},
they show that the helicoid is the unique non-flat complete,
embedded minimal disk.  They go on to discuss how the techniques
of their proof might allow one to show something similar to
Theorem \ref{ThirdMainThm} for surfaces in $\mathcal{E}(1,+)$ and
the implications this has for the possible conformal structure of
complete embedded minimal surfaces in $\Real^3$. They do this
without going into the details or addressing the difficulties, but
indicate how such a statement may be proved using the ideas and
techniques of their proof.  In \cite{BB1}, we more directly use the
geometric structure given by Colding and Minicozzi for embedded
minimal disks to prove the uniqueness of the helicoid.  In this
paper, we generalize our argument to surfaces in
$\mathcal{E}(1,+)$, thus determining the asymptotic structure of
all elements of $\mathcal{E}(1)$.

Recall the Weierstrass representation takes a triple $(M,g, dh)$
where $M$ is a Riemann surface, $g$ is a meromorphic function and
$dh$ is a meromorphic one form (which has a zero everywhere $g$
has a pole or zero) that satisfy certain natural compatibility
conditions, and gives a minimal immersion of $M$ into $\Real^3$:
\begin{equation}\label{WeierstrassRep}
\mathbf{F}:=\re \int\left(\frac{1}{2}
(g^{-1}-g),\frac{i}{2}(g^{-1}+g), 1\right) dh.
\end{equation}
Moreover, the immersion $\mathbf{F}$ is such that $\re dh= \mathbf{F}^*dx_3$ and $g$ is the stereographic projection of the Gauss map of the image of $\mathbf{F}$. Any immersed minimal surface in $\Real^3$ admits such a
representation. For the helicoid with $z \in \mathbb{C}$ one has:
\begin{equation}g:=e^{i\alpha z}; \;\;\; dh:=dz;\;\;\; \alpha \in \Real^+.
\end{equation} Notice that on the helicoid both $\frac{dg}{g}$ and
$dh$ have double poles at infinity; moreover,
$\frac{dg}{g}-i\alpha dh$ is identically zero. For $\Sigma \in
\mathcal{E}(1)$, Theorem \ref{ThirdMainThm}, the Weierstrass
representation, and embeddedness immediately imply that near the
puncture the Weierstrass data is asymptotic to that of a helicoid.
This is an immediate consequence of Theorem \ref{ThirdMainThm}
above and Theorem 2 of \cite{HPR}, though we present our own
proof in Section \ref{pfsec} (where we also prove Theorem
\ref{ThirdMainThm}).  Indeed, we have:
\begin{cor}\label{ThirdMainThmCor}  For $\Sigma$ as in Theorem \ref{ThirdMainThm},
there exists an $\alpha\in\Real$ so $\frac{dg}{g}-i\alpha dh$
holomorphically extends over the puncture.  Equivalently, after possibly translating parallel to the
$x_3$-axis, in an appropriately chosen neighborhood of the puncture,
$\Gamma\subset \Sigma$, $g(p)=\exp( {i\alpha z(p)}+F(p))$ where $F:\Gamma\to
\mathbb{C}$ extends holomorphically over the puncture with a zero
there and $z=x_3+ix_3^*$ is a holomorphic coordinate on $\Gamma$.
Here $x_3^*$ is the harmonic conjugate of $x_3$.
\end{cor}


Theorem 1 of \cite{HPR} implies that for Weierstrass data
$(\Gamma,g,dh)$ as in Corollary \ref{ThirdMainThmCor}, that also
satisfy a certain flux condition, the Weierstrass representation
gives a minimal surface that is $C^0$-asymptotic to a (vertical)
helicoid $H$. That is, for any $\epsilon>0$, there exists
$R_\epsilon>0$, so that $\Gamma\backslash B_{R_\epsilon}(0)$ has
Hausdorff distance to $H\backslash B_{R_\epsilon}(0)$ less than
$\epsilon$.  For elements of $\mathcal{E}(1)$, as the Weierstrass
data is defined on a surface with only one end, this flux condition
is automatically satisfied by Stokes' theorem.  Thus, Theorem
\ref{ThirdMainThm} allows one to immediately apply Theorem 1 of
\cite{HPR} and obtain:
\begin{cor} \label{AsympHel}
If $\Sigma\in \mathcal{E}(1)$ is non-flat then $\Sigma$ is
$C^0$-asymptotic to some helicoid.
\end{cor}

Theorem \ref{ThirdMainThm} and its corollaries
complete the classification of the conformal type and asymptotic geometry of complete embedded minimal surfaces in $\Real^3$
with finite topology. Indeed, let  $\mathcal{E}(k)$ be the space of complete, embedded minimal
surfaces with finite topology and $k$ ends.  Meeks and Rosenberg, in Corollary 1.1 of \cite{MR2}, completely classify the conformal type of these surfaces when $k\geq 2$.  Indeed, they show such a surface is conformal to a compact
Riemann surface with $k$ punctures.  However, they can only describe the asymptotic geometry at $k-2$ of the ends. Collin, in Theorem 1 of \cite{Co}, overcomes this obstacle by proving that all such surfaces have finite total curvature. This, together with classic results of Huber \cite{Huber} and Osserman \cite{OSS}, recovers not only the conformal type of surfaces in $\mathcal{E}(k)$ for $k\geq 2$, but also gives a description of their asymptotic geometry. Thus, he completes the classification for embedded minimal surfaces of finite topology and two or more ends. Combined with Theorem \ref{ThirdMainThm}, we then have the following classification result for any minimal surface of finite topology that is complete and embedded in $\Real^3$:
\begin{cor}\label{classifcor}
Let $\Sigma \in \mathcal{E}(k), k \geq 1$. Then $\Sigma$ is
conformal to a punctured compact Riemann surface.  Moreover, if $k
\geq 2$, then $\Sigma$ has finite total curvature and each end of
$\Sigma$ is asymptotic to either a plane or a catenoid.  If $k = 1$,
then either $\Sigma$ is a plane or it has infinite total curvature
and its end is asymptotic to a helicoid.
\end{cor}

Let us now recall the argument of \cite{BB1}, where we provide an
alternative proof to the uniqueness of the helicoid. There it is shown that any
complete, non-flat, properly embedded minimal disk can be
decomposed into two regions: one a region of strict spiraling,
i.e. the union of two strictly spiraling multi-valued graphs over
the $x_3=0$ plane (after a rotation of $\Real^3$), and the other a
neighborhood of the region where the graphs are joined and where
the normal has small vertical component. By strictly spiraling, we
mean that each sheet of the graph meets any  (appropriately
centered) cylinder with axis parallel to the $x_3$-axis in a curve
along which $x_3$ strictly increases (or decreases). This
follows from existence results for multi-valued minimal graphs in
embedded disks found in \cite{CM2} and an approximation result for
such minimal graphs from \cite{EXC}. The strict spiraling is then
used to see that $\nabla_\Sigma x_3\neq 0$ everywhere on the
surface; thus, the Gauss map is not vertical and the holomorphic
map $z=x_3+ix_3^*$ is a holomorphic coordinate. By looking at the
$\log$ of the stereographic projection of the Gauss map, the
strict spiraling is used to show that $z$ is actually a proper map
and thus, conformally, the surface is the plane. Finally, this
gives strong rigidity for the Weierstrass data, implying the
surface is a helicoid.

For $\Sigma \in \mathcal{E}(1,+)$, as there is finite genus and only one end, the
topology of $\Sigma$ lies in a ball in $\Real^3$, and so,
by the maximum principle, all components of the intersection of
$\Sigma$ with a ball disjoint from the genus are disks.  Hence,
outside of a large ball, one may use the local results of
\cite{CM1,CM2,CM3,CM4} about embedded minimal disks. In \cite{BB1},
the trivial topology of $\Sigma$ allows one to deduce global
geometric structure immediately from these local results. For
$\Sigma \in \mathcal{E}(1,+)$, the presence of non-zero genus
complicates matters.  Nevertheless, the global structure will follow
from the far reaching description of embedded minimal surfaces given
by Colding and Minicozzi in \cite{CM5}.
In particular, as $\Sigma$ has one end, globally it looks like a
helicoid (see Appendix \ref{GlobGeomApp}). Following \cite{BB1}, we
first prove a sharper description of the global structure (in
Section \ref{DecSec}); indeed, one may generalize the decomposition
of \cite{BB1} to $\Sigma \in \mathcal{E}(1,+)$ as:
\begin{thm}
\label{FirstMainThm} There exist $\epsilon_0>0$ and a decomposition (see Figure \ref{GenusDecomp})
of $\Sigma$ into disjoint subsets $\mathcal{R}_A$, $\mathcal{R}_S$,
and $\mathcal{R}_G$ such that:
\begin{enumerate}
\item
$\mathcal{R}_G$ is compact, connected, has connected boundary and
$\Sigma\backslash \mathcal{R}_G$ has genus 0;
\item after a rotation
of $\Real^3$, $\mathcal{R}_S$ can be written as the union of two
(oppositely oriented) strictly spiraling multi-valued graphs $\Sigma^{1}$ and $\Sigma^{2}$;
\item in $\mathcal{R}_A$, $|\nabla_\Sigma
x_3|\geq \epsilon_0$.
\end{enumerate}
\end{thm}

\begin{figure}
 \includegraphics[width=3in]{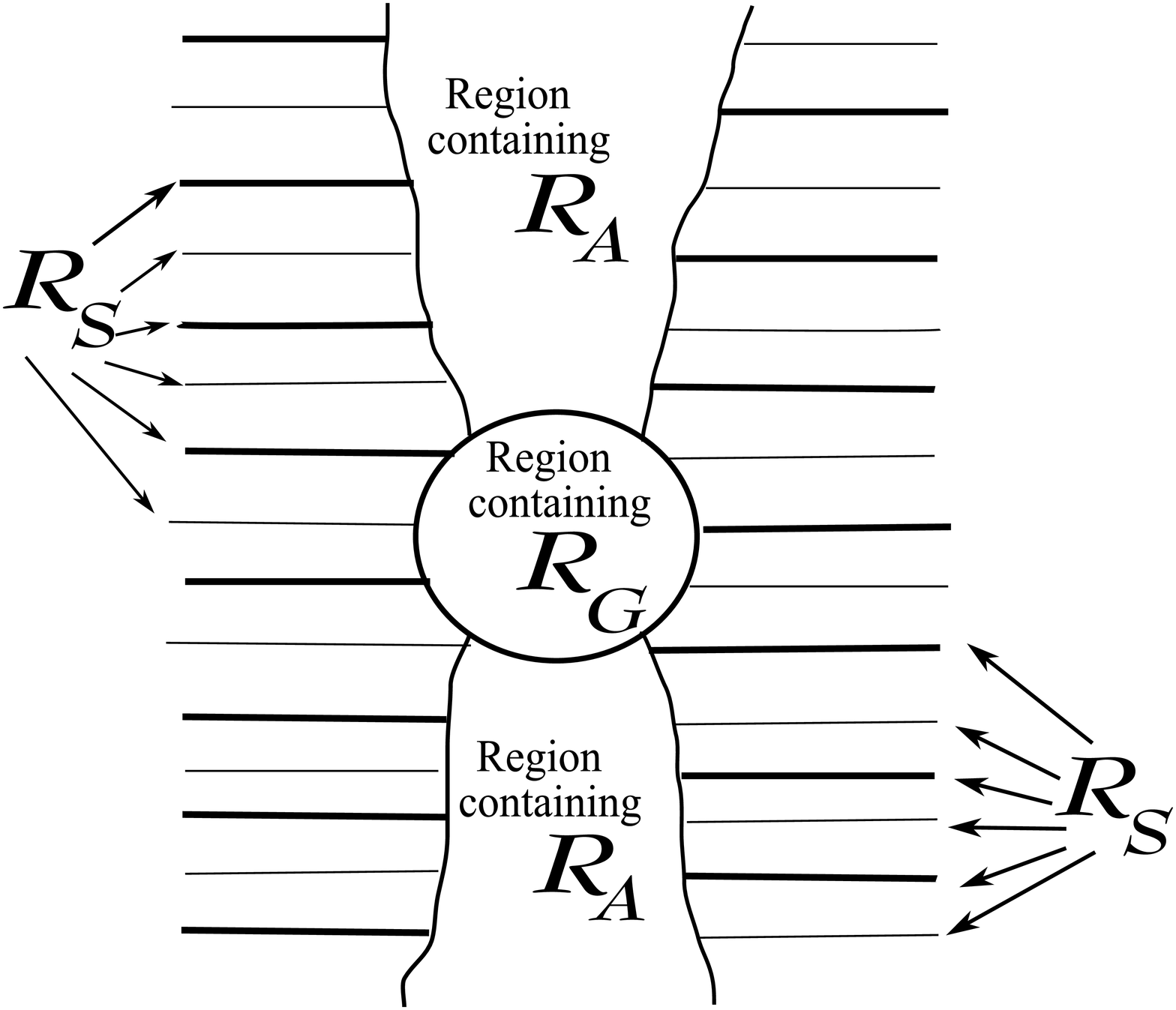}
\caption{A cross-sectional sketch of the three regions in the
decomposition of $\Sigma$ as outlined in Theorem
\ref{FirstMainThm}.}\label{GenusDecomp}
\end{figure}

\begin{rem} \label{MultiGraphRem}
We say $\Sigma^{i}$ ($i=1,2$) is a multi-valued graph if it can be
decomposed into $N$-valued $\epsilon$-sheets (see Definition
\ref{epsilonsht}) with varying center.  That is,
$\Sigma^i=\cup_{j=-\infty}^\infty \Sigma_j^i$ where each
$\Sigma_j^i=y^i_j+\Gamma_{u^i_j}$ is an $N$-valued $\epsilon$-sheet.
Strict spiraling is then equivalent to $(u_j^i)_\theta\neq 0$ for
all $j$. A priori, the axes of the multi-valued graphs vary, a fact
that introduces additional book-keeping.  For the sake of clarity,
we assume that each $\Sigma^i$ is an $\infty$-valued
$\epsilon$-sheet -- i.e $\Sigma^i$ is the graph, $\Gamma_{u^i}$, of
a single function $u^i$ with $u^i_\theta\neq 0$.
\end{rem}
To prove this decomposition, we first find the region of strict
spiraling, $\mathcal{R}_S$.  The strict spiraling controls the
asymptotic behavior of level sets of $x_3$ which, as $x_3$ is
harmonic on $\Sigma$,  gives information about $x_3$ in all of
$\Sigma$. More precisely:
\begin{prop} \label{SecondMainThm}
There exists $\Gamma\subset \Sigma$, an annulus, so that
$\Sigma\backslash \Gamma$ is compact and such that: In $\Gamma$,
$\nabla_\Sigma x_3\neq 0$ and, for all $c\in \Real$, $\Gamma\cap
\set{x_3=c}$ consists of either one smooth, properly embedded curve
or two smooth, properly embedded curves each with one endpoint on
$\partial\Gamma$ along with a finite number of smooth curves with both endpoints on $\partial \Gamma$.  Moreover, if $c$ is a regular value of $x_3$ then $\Gamma\cap \set{x_3=c}$ is a subset of the unbounded component of $\Sigma \cap \set{x_3=c}$.
\end{prop}
The decomposition allows us to argue as in \cite{BB1}, though the
non-trivial topology again adds some technical difficulties.  By
Stokes' Theorem, $x_3^*$ (the harmonic conjugate of $x_3$) exists on
$\Gamma$ and thus there is a well defined holomorphic map
$z:\Gamma\to \mathbb{C}$ given by $z=x_3+ix_3^*$. Proposition
\ref{SecondMainThm} implies that $z$ is a holomorphic coordinate on
$\Gamma$.  We claim that $z$ is actually a proper map and so
$\Gamma$ is conformally a punctured disk. Following \cite{BB1}, this
can be shown by studying the Gauss map. On $\Gamma$, the
stereographic projection of the Gauss map, $g$, is a holomorphic map
that avoids the origin. Moreover, the minimality of $\Sigma$ and the
strict spiraling in $\mathcal{R}_S$ imply that the winding number of
$g$ around the inner boundary of $\Gamma$ is zero. Hence, by
monodromy there exists a holomorphic map $f:\Gamma\to \mathbb{C}$
with $g=e^f$. Then, as in \cite{BB1}, the strict spiraling in
$\mathcal{R}_S$ imposes strong control on $f$ which is sufficient to
show that $z$ is proper. Further, once we establish $\Gamma$ is
conformally a punctured disk, the properties of the level sets of
$f$ imply that it extends meromorphically over the puncture with a
simple pole. This gives Theorem \ref{ThirdMainThm}.

\section{Structural Properties of $\Sigma$}
In the next five subsections, we develop the tools needed to prove
the decomposition of Theorem \ref{FirstMainThm} and Proposition
\ref{SecondMainThm}.  Many of these are extensions of those
developed for the simply connected case, which can be found in
Section 2 of \cite{BB1}.
\subsection{Preliminaries}
We first introduce some notation.  Unless otherwise specified, throughout the paper let
$\Sigma\in \mathcal{E}(1,+)$, i.e. $\Sigma$ is a complete,
embedded minimal surface with finite and positive genus, $k$, and
one end. Here we say that a surface has genus $k$ if it is
homeomorphic to a compact, oriented genus $k$ surface with at most
a finite number of punctures.  As $\Sigma$ has one end and is
complete in $\Real^3$, there exists an $R>0$ so that one of the
components $\overline{\Sigma}$ of $\Sigma\cap B_R$ is a compact
surface with connected boundary and genus $k$; we refer to
$\overline{\Sigma}$ as the \emph{genus of $\Sigma$}. Thus, $\Sigma
\backslash \overline{\Sigma}$ has genus $0$ and is a neighborhood
of the end of $\Sigma$.  By homothetically rescaling, we may assume that the
genus, $\overline{\Sigma}$, lies in $B_1$ and
$\sup_{\overline{\Sigma}} |A|^2 \geq 1$. Here $|A|^2$ denotes the
norm squared of the second fundamental form of $\Sigma$ and
$B_r(y)$ represents the Euclidean ball of radius $r$, centered at
$y$, whereas $\mathcal{B}_r(y)$ denotes the intrinsic ball in
$\Sigma$ of radius $r$ centered at $y$.
Throughout the paper, when we say \textit{far from (or near) the
genus}, we mean extrinsically with respect to this scale.  That
is, a subset of $\Real^3$ is far from the genus if the Euclidean
distance to $B_1$ is large.

Denote by $\Pi:\Real^3\to \Real^2$ the projection
$\Pi(x_1,x_2,x_3)=(x_1,x_2)$.  Let
\begin{equation}
\mathbf{C}_\delta (y)=\set{x: (x_3-y_3)^2 \leq \delta^2
\big((x_1-y_1)^2 +(x_2-y_2)^2\big)}\subset \Real^3
\end{equation}
be the complement of a cone and set
$\mathbf{C}_\delta=\mathbf{C}_\delta (0)$. Given a real-valued
function, $u$, defined on a domain $\Omega\subset \Real^+\times
\Real$, define the map $\Phi_u :\Omega\to \Real^3$ by $\Phi_u
(\rho,\theta)=(\rho \cos \theta,\rho \sin \theta, u(\rho,\theta))$
so the image is a multi-valued graph. A natural domain is the polar
rectangle:
\begin{equation}
\pRect{r_1}{r_2}{\theta_1}{\theta_2}=\set{(\rho,\theta) \mid r_1\leq
\rho\leq r_2, \theta_1\leq \theta \leq \theta_2}.
\end{equation}
Indeed, for $u$ defined on $\pRect{r_1}{r_2}{\theta_1}{\theta_2}$,
$\Phi_u(\pRect{r_1}{r_2}{\theta_1}{\theta_2})$ is a multi-valued
graph over the annulus $D_{r_2}\backslash D_{r_1}$. Thus, $\Gamma_u
:= \Phi_u (\Omega)$ is the graph of $u$, and $\Gamma_u$ is embedded
if and only if $w\neq 0$, where the separation $w$ of $u$ is defined
as $w(\rho,\theta)=u(\rho,\theta+2\pi)-u(\rho,\theta)$. We say a
multi-valued graph, $\Gamma_u$, \textit{strictly spirals} if, for
$u:\Omega \to \Real^3$, $u_\theta \neq 0$.  Note at times we fail to
distinguish between $u$ and its graph $\Gamma_u$ -- the meaning will
be clear from context.

Recall that $u$ satisfies the
minimal surface equation if:
\begin{equation} \label{MinSurfEq}
\Div \left(\frac{\nabla u}{\sqrt{1+|\nabla
u|^2}}\right)=0.
\end{equation}
The graphs of interest to us will satisfy
the following flatness condition:
\begin{equation}
\label{MainInEq21}
|\nabla u|+\rho |\Hess_u|+4 \rho \frac{|\nabla w|}{|w|}+\rho^2 \frac{|\Hess_w|}{|w|}\leq \epsilon <\frac{1}{2\pi}.
\end{equation}
As multi-valued minimal graphs are fundamental to the description of
the asymptotic behavior, we introduce some notation for them.
\begin{defn}\label{weakepsilonsht}  A multi-valued minimal graph $\Sigma_0$ is a \emph{weak} \emph{$N$-valued} (\emph{$\epsilon$}-)\emph{sheet} (\emph{centered at $y$ on the scale $s$}), if $\Sigma_0=\Gamma_{u}+y$ and $u$, defined on
$\pRect{s}{\infty}{-\pi N}{ \pi N}$, satisfies  \eqref{MinSurfEq},
has $|\nabla u|\leq \epsilon$, and
$\Sigma_0\subset \mathbf{C}_{\epsilon}(y)$.
\end{defn}
We will often need more control on the sheets as well as a normalization at $\infty$:
\begin{defn} \label{epsilonsht} A multi-valued minimal graph $\Sigma_0$ is an (\emph{strong}) \emph{$N$-valued} (\emph{$\epsilon$}-)\emph{sheet} (\emph{centered at $y$ on the scale $s$}), if $\Sigma_0=\Gamma_{u}+y$ is a weak $N$-valued $\epsilon$-sheet centered at $y$ on scale $s$, and in addition $u$ satisfies
\eqref{MainInEq21} and $\lim_{\rho\to \infty} \nabla u(\rho,0)=0$.
\end{defn}
Using Simons' inequality, Corollary 2.3 of \cite{MMGPD} shows that
on the one-valued middle sheet of a 2-valued graph satisfying
$\eqref{MainInEq21}$, the hessian of $u$ has faster than linear
decay. This implies a Bers like
result on asymptotic tangent planes -- i.e. the normalization at $\infty$ in the definition of $\epsilon$-sheet is well defined. Indeed, for $\Gamma_u$ a 2-valued $\epsilon$-sheet, one has gradient decay,
\begin{equation}
  \label{GradDecay}
  |\nabla u|(\rho,0) \leq  C\epsilon \rho^{-5/12}.
\end{equation}
Note that Colding and Minicozzi show, using standard elliptic theory, that for sufficiently large $N$ and small $\delta$ a weak $N$-valued
$\delta$-sheet contains a $4$-valued $\epsilon$ sheet as a sub-graph.
For the details, we refer the reader to Proposition 2.3 of \cite{BB1}.

Finally, as in the papers of Colding and Minicozzi, we are
interested in points with large curvature relative to nearby points,
as around these points multi-valued graphs form (see \cite{CM2}).
The precise definition we use is the following:
\begin{defn} The pair $(y,s)$, $y\in \Sigma$, $s>0$, is a (\emph{$C$}) \emph{blow-up pair} on scale $s$ if
\begin{equation}
\sup_{B_s(y)\cap \Sigma} |A|^2\leq 4|A|^2(y)=4C^2 s^{-2}.
\end{equation}
\end{defn}
\begin{rem}
The constant $C$ will be specified in some of the theorems but it
should always be thought of as being very large.
\end{rem}

\subsection{Existence of Multi-valued Graphs}\label{StrSec}
To obtain the decomposition of Theorem \ref{FirstMainThm} we will
need two propositions regarding the large scale geometric structure
of elements of $\mathcal{E}(1,+)$.  These generalize results for
disks from \cite{CM1} and \cite{CM2} on the existence and
extendability of multi-valued graphs in embedded minimal disks. It
should be noted that many of the proofs of \cite{CM1,CM2} did not
require that the surface be a disk but only that the boundary be
connected, a fact used in \cite{CM5} to extend the description of
embedded minimal disks of \cite{CM1,CM2} to finite genus surfaces.
The first result we will need gives the existence of an $N$-valued
graph starting near the genus and extending as a graph all the way
out; i.e. the initial graph is a subset of a graph over an unbounded
annulus. The second result is similar but applies for a blow-up pair
far from the genus. Namely, for such a pair a multi-valued graph
forms on the scale of the pair and extends as a graph all the way
out.  It may be helpful to compare with the analogous results for
disks, i.e. Theorem 0.3 of \cite{CM1} and Theorem 0.4 of \cite{CM2}.


Note that local versions of the propositions stated below are used
by Colding and Minicozzi in \cite{CM5}, specifically in the proof of
their compactness result for finite genus surfaces (i.e. Theorem 0.9
of \cite{CM5}). However, they are not explicitly stated in
\cite{CM5}. Thus, for the sake of completeness, we include a proof
of both propositions, {assuming} Theorem 0.9 of \cite{CM5}, in
Appendix \ref{LamResApp}. Both propositions require a rotation of
$\Real^3$.  However, because both propositions come from the global
geometric structure of $\Sigma$, the rotations are the same.
\begin{prop} \label{GraphNrConePrp}Given
$\epsilon>0$ and $N\in \mathbb{Z}^+$ there exists an $R>1$ so that:
After a rotation of $\Real^3$ there exists $\Sigma_g\subset \Sigma$,  a weak $N$-valued $\epsilon$-sheet centered at $0$ and on scale $R$.
\end{prop}

\begin{prop} \label{GraphExtPrp} Given $\epsilon>0$ sufficiently small and $N\in \mathbb{Z}^+$ there exist $C_1>0$
and $R>1$ so: After a rotation of $\Real^3$, if $(y,s)$ is a $C_1$
blow-up pair in $\Sigma$ and $|y|\geq R$ then there exists
$\Sigma_g\subset \Sigma$, a weak $N$-valued $\epsilon$-sheet
centered at $y$ and on scale $s$.
\end{prop}

\subsection{Global Structure of $\Sigma$}
For an element $\Sigma \in \mathcal{E}(1,0)$ (i.e. a minimal disk),
the existence of a weak $\epsilon$-sheet (i.e. the existence of a
blow-up pair) allowed one to immediately appeal to the one-sided
curvature estimates of \cite{CM4} (see Appendix \ref{oscsec}). This
allowed one to deduce important information about the global
structure of $\Sigma$. In particular, one had a type of
``regularity'' for the set of blow-up pairs; that is, all the
blow-up pairs of $\Sigma$ were forced to lie within a wide cone. A
related property was that if $(0,1)$ is a blow-up pair in $\Sigma$
then blow-up pairs far from $0$ have scale a small fraction of the
distance to $0$.

We will need a similar results for $\Sigma\in \mathcal{E}(1,+)$;
however, because the one-sided curvature estimate is very sensitive
to the topology, the non-trivial genus will introduce some technical
difficulties. We discuss how to overcome these in Appendix
\ref{oscsec}.  As a consequence, we have the following lemma, which
asserts that there is a cone (centered at the origin) so that for
blow-up pairs far from the genus, the pair must lie within the cone,
i.e. we recover the ``regularity'' of the set of blow-up pairs.  We
point out that this result is particularly useful when combined with
Corollary \ref{OscCor} and Proposition \ref{GraphExtPrp}, as the
three imply that for small $\delta'$ and a blow-up pair $(y,s)$ far
from the genus, one may apply the one-sided curvature estimate in
$\mathbf{C}_{\delta'}(y)$ exactly as was done in the case for disks.
\begin{lem}\label{BUPinConeLem}
 There exists a $\delta>0$ and $R>1$ so that if $(y,s)$ is a blow-up pair in $\Sigma$ and $|y|\geq R$ then $y\notin \mathbf{C}_{\delta}(0)$.
\end{lem}
\begin{proof}
Fix $\epsilon=1/2$, and let $\delta_0$ be the value given by
Corollary \ref{OscCor0}.  By Proposition \ref{GraphNrConePrp}, $\Sigma$ contains a weak 2-valued $\delta_0$-sheet centered
at $0$ and with scale $R_0>1$. Thus, as $\overline{\Sigma}\subset
B_1\subset B_{R_0}$, we may apply Corollary \ref{OscCor0} to deduce
that in the set $\mathbf{C}_{\delta_0}\backslash B_{2R_0}$ every component of $\Sigma$ is a graph with gradient bounded
by $1/2$. In particular, there are no blow-up pairs in this set.
Thus, we may take $R=2R_0$ and $\delta=\delta_0$.
\end{proof}

The second global result for disks also generalizes.  Indeed, we
claim that the further a blow-up pair is from the genus, the smaller
the ratio between the scale and the distance to the genus. This is
an immediate consequence of the control on curvature around blow-up
pairs as given by Proposition \ref{genuscbprop} (an extension of
Lemma 2.26 of \cite{CY} to $\Sigma$).  Indeed, for blow-up pairs far
from the genus, the scale is small relative to distance to the
genus:
\begin{cor}
 \label{ConditiononBUPsTHM}  Given $\alpha , C_1>0$ there exists an $R$ such that for
 $(y,s)$, a $C_1$ blow-up pair of $\Sigma$ with $|y| \geq R$ then $s <\alpha|y|$.
\end{cor}
\begin{proof}
Recall we have normalized $\Sigma$ so $\sup_{B_1 \cap \Sigma}
|A|^2\geq 1$. Now suppose the result did not hold. Then there
exists a sequence $(y_j,s_j)$ of $C_1$ blow-up pairs with $|y_j|
\geq j$ and $s_j\geq\alpha |y_j|$.  Set $K_1 = 2/\alpha$. By
Proposition \ref{genuscbprop} there exists $K_2$ such that
$\sup_{B_{K_1s_j}(y_j)\cap \Sigma}|A|^2 \leq K_2s_j^{-2}$.  Since
$B_1\subset B_{K_1 s_j}(y_j)$, $\sup_{B_1\cap \Sigma} |A|^2 \leq
K_2s_j^{-2}$.  But $s_j \geq \alpha|y_j|\geq \alpha j$; thus for
$j$ sufficiently large one obtains a contradiction.
\end{proof}

\subsection{Blow-up Sheets}
In order to get the strict spiraling in the decomposition of
Theorem \ref{FirstMainThm}, we need to check that the multi-valued
graphs that make up most of $\Sigma$ can be consistently
normalized. To that end, we note that, for blow-up pairs far enough
from the genus, one obtains a nearby $\epsilon$-sheet (i.e. we have
a normalized multi-valued graph). This is essentially Theorem 2.5 in \cite{BB1}.

\begin{thm}
\label{InitShtExst} Given $\epsilon>0$, $N\in \mathbb{Z}^+$, there
exist $C_1, C_2>0$ and $R>0$ so:  Suppose that $(y,s)$ is a $C_1$
blow-up pair of $\Sigma$ with $|y| > R$.  Then there exists (after a
rotation of $\Real^3$) an $N$-valued $\epsilon$-sheet
$\Sigma_1=y+\Gamma_{u_1}$ centered at $y$ on scale $s$.
  Moreover,
the separation over $\partial D_s(\Pi(y))$ of $\Sigma_1$ is bounded
below by $C_2 s$.
\end{thm}

The proof of the theorem is exactly the same as the proof of Theorem
2.5 in \cite{BB1} with one modification. Where the proof of
\cite{BB1} uses Theorem 0.2 of \cite{CM2} to produce a weak
$N_0$-valued sheet ($N_0$ is determined in the proof), one must now
use Proposition \ref{GraphExtPrp}. Thus, in the above hypothesis,
the blow-up pair must satisfy the additional criteria of $|y|> R$ so
that one may appeal to Proposition \ref{GraphExtPrp}.

Following Colding and Minicozzi, we need to next understand the structure of $\Sigma$
between the sheets of this initial multi-valued graph, $\Sigma_1$.
We claim that in between this sheet, $\Sigma$ consists of
exactly one other $\epsilon$-sheet.  To make this more precise, suppose $u$ is defined on
$\pRect{1/2}{\infty}{-\pi N-3\pi}{\pi N+3\pi}$ and $\Gamma_u$ is
embedded.  We define $E$ to be the region over $D_\infty \backslash
D_1$ between the top and bottom sheets of the concentric sub-graph of
$u$.  That is:
\begin{multline}
E=\{(\rho\cos \theta, \rho\sin \theta, t): \\ 1\leq \rho\leq \infty,
-2\pi \leq \theta <0, u(\rho,  \theta-\pi N)<t<u(\rho, \theta+
(N+2)\pi \}.
\end{multline}
When $\Sigma$ is a disk, Colding and Minicozzi in Theorem I.0.10 of
\cite{CM4} show that $\Sigma \cap E \backslash \Sigma_1$ consists of
a single graphical piece.  Thus, using $\Sigma_1$ and the one-sided
curvature estimate of \cite{CM4}, the gradient of this second
graphical component is controlled.  As before, when there are enough
sheets in this second multi-valued graph and the gradient is
controlled, standard elliptic theory establishes \eqref{MainInEq21}
on a sub-graph and hence one obtains two $\epsilon$-sheets spiraling
together. We refer the reader to Theorem 2.6 of \cite{BB1} for the
details.  In the more general setting of this paper, as long as the
part of $\Sigma$ between the sheets of $\Sigma_1$ makes up a second
minimal graph and we can apply the one-sided curvature estimates,
the proof of Theorem 2.6 of \cite{BB1} applies.  Thus, we must
verify both the existence of this second multi-valued graph and that
we are able to apply the one-sided curvature estimate to it.  By
patching together two results of Colding and Minicozzi from
\cite{CM4} the first issue is easily handled.   The global structure
of $\Sigma$, in particular Lemma \ref{BUPinConeLem}, implies that as
long as the blow-up pair is far from the genus there is no problem
handling the second issue either:
\begin{thm} \label{TwoInitShtExstThm}
Given $\epsilon>0$ sufficiently small there exist $C_1,C_2>0$ and
$R>1$ so: Suppose $(y,s)$ is a $C_1$ blow-up pair, with $|y|>R$.
Then there exist two $4$-valued $\epsilon$-sheets
$\Sigma_i=y+\Gamma_{u_i}$ ($i=1,2$) on the scale $s$ centered at $y$
which spiral together (i.e. $u_1(s,0)< u_2(s,0)<u_1(s,2\pi)$).
Moreover, the separation over $\partial D_s(\Pi(y))$ of $\Sigma_i$
is bounded below by $C_2 s$.
\end{thm}
\begin{rem} We refer to $\Sigma_1, \Sigma_2$ as (\emph{$\epsilon$-})blow-up sheets \emph{associated
with}
$(y,s)$.
\end{rem}
\begin{proof}
Let $\delta>0$ and $R>1$ be given by Lemma \ref{BUPinConeLem}. Using
this $\delta$ and $\epsilon/2$, pick $\delta_0<\epsilon/2$ as in
Corollary \ref{OscCor} (and increase $R$ if needed).  Theorem
\ref{InitShtExst} gives one $\tilde{N}$-valued $\delta_0$-sheet,
$\Sigma_1$, forming near $(y,s)$ for appropriately chosen $C_1$ (and
possibly after again increasing $R$). Here we choose the
$\tilde{N}>4$ as in Theorem 2.6 of \cite{BB1} -- this
allows one to establish \eqref{MainInEq21} on a sub-graph of the
second graph.  Indeed, once we establish that $E$, the region between
the sheets of $\Sigma_1$, is a weak $(\tilde{N}-4)$-valued $\epsilon/2$-sheet the argument of Theorem 2.6 of \cite{BB1} carries over unchanged.

We now show that $\Sigma \cap E \backslash \Sigma_1$ consists of
exactly one multi-valued graph.  Theorem I.0.10 of \cite{CM4}
implies that near the blow-up pair the part of $\Sigma$ between
$\Sigma_1$ is a $\tilde{N}-4$ sheeted graph $\Sigma_2^{in}$; i.e. if
$R_0$ is chosen so $B_{4R_0}(y)$ is disjoint from the genus then
$B_{R_0}(y)\cap E\cap\Sigma\backslash \Sigma_1=\Sigma_2^{in}$. To
ensure $\Sigma_2^{in}$ is non-empty, we increase $R$ so that
$|y|\geq 8s$ (which we may do by Corollary
\ref{ConditiononBUPsTHM}). On the other hand, Appendix D of
\cite{CM4} guarantees that, outside of a very large ball centered at
the genus, the part of $\Sigma$ between $\Sigma_1$ is a
$\tilde{N}-4$ sheeted graph, $\Sigma_2^{out}$. That is, for $R_1\geq
|y|$ large, $E\cap \Sigma\backslash (B_{R_1}\cup
\Sigma_1)=\Sigma_2^{out}$.  By our choice of $\delta_0,R$, we can
now apply the one-sided curvature estimate in $E$, and so all the
components of $E\backslash \Sigma_1$ are graphs with gradient
bounded by $\epsilon/2$. Thus, it suffices to show that
$\Sigma_2^{in}$ and $\Sigma_2^{out}$ are subsets of the same
component.  If this was not the case, then, as $\Sigma_2^{in}$ is a
graph and $\Sigma$ is complete, $\Sigma_2^{in}$ must extend inside
$E$ beyond $B_{R_1}$.  But this contradicts Appendix D of \cite{CM4}
by giving two components of $\Sigma\backslash \Sigma_1$ in $E\cap
\Sigma\backslash B_{R_1}$.
\end{proof}

\subsection{Existence of Blow-Up Pairs}
While the properties of $\epsilon$-sheets will give the strictly
spiraling region of $\Sigma$, $\mathcal{R}_S$, to understand the
region where these sheets fit together (i.e. what will become
$\mathcal{R}_A$), we need a handle on the distribution of the
blow-up pairs of $\Sigma$. Notice that the global structural results
discussed above, i.e. Lemma \ref{BUPinConeLem} and Corollary
\ref{ConditiononBUPsTHM}, give weak information of this sort.

In the case of trivial topology -- i.e.  Theorem 2.8 of \cite{BB1}
-- non-flatness gives one blow-up pair $(y_0,s_0)$, which in turn
yields associated blow-up sheets. Then by Corollary III.3.5 of
\cite{CM3} coupled with the one-sided curvature estimate, the
blow-up sheets give the existence of nearby blow-up pairs $(y_{\pm
1},s_{\pm 1})$ above and below $(y_0,s_0)$ (see also Lemma 2.5 of
\cite{CY}). Iterating, one determines a sequence of blow-up pairs
that are then used to construct the decomposition. The extension of
the argument to surfaces in $\mathcal{E}(1,+)$ is much the same,
though again there are various technical difficulties complicating
matters. Essentially, the proof will rely on three things.  First,
the result of \cite{CM3} is local; it depends on the topology being
trivial in a large ball relative to the scale of the blow-up pair.
Second, by Lemma \ref{BUPinConeLem} and Corollary \ref{OscCor}, for
blow-up pairs sufficiently far from the genus, we can apply the
one-sided curvature estimate. Thus, we conclude that points of large
curvature near a blow-up pair must lie within a cone with vertex the
point of the blow-up pair. As a consequence, blow-up pairs can be
constructed that are truly above (or below) a given blow-up pair.
Third, Corollary \ref{ConditiononBUPsTHM} implies that the scale of
blow-up pairs far from the genus is small relative to this distance.

Thus, it will suffice to find two blow-up pairs far from the
genus in $\Sigma$, one above and one below the genus.  We first verify this is possible:
\begin{lem}\label{firstbupprop}
Given $\epsilon>0$ sufficiently small and $h>1$, $C_1>0$, there
exist pairs $(y_\pm,s_\pm)$ such that $(y_\pm,s_\pm)$ are $C_1$
blow-up pairs of $\Sigma$ and $x_3(y_+)>h>-h>x_3(y_-)$.
\end{lem}
\begin{proof} Fix a $\delta_0>0$ small, it will be specified in what follows.
Proposition \ref{GraphNrConePrp} of this paper and Appendix D of
\cite{CM4} together guarantee the existence of two
$\tilde{N}$-valued graphs spiraling together over an unbounded
annulus (with inner radius $\overline{R}$) and lying in
$\mathbf{C}_{\delta_0}$. Moreover, by construction, one of these
is a weak $\delta_0$-sheet.  By using Corollary \ref{OscCor0}, and
replacing $\overline{R}$ by $2\overline{R}$ we can control the
gradient on both $\tilde{N}$-valued graphs. As before, for large
enough $\tilde{N}$ and sufficiently small $\delta_0$, we get
\eqref{MainInEq21} on a sub-graph of both graphs.
 Because Proposition \ref{GraphNrConePrp} already provides the necessary rotation, we get two $N$-valued $\epsilon$-sheets around the genus,
$\Sigma_1,\Sigma_2$, on some scale $\tilde{R}$ and in
$\mathbf{C}_\epsilon$.  We may make $\epsilon$ as small as we like
by shrinking $\delta_0$.

Theorem III.3.1 of \cite{CM3} is the analogue to Corollary III.3.5
of \cite{CM3} for minimal surfaces with connected boundary. Thus,
for any $r_0 \geq \max \{1, \tilde{R}\}$, Theorem III.3.1 implies
there is large curvature above and below the genus at points
$x_\pm$. Precisely, there exist $x_\pm \in \Sigma \backslash
B_{4r_0}$ such that $|x_\pm|^2|A|^2(x_\pm) \geq 4C_1^2$. Hence, by
a standard blow-up argument (see Lemma 5.1 of \cite{CM2}), one
gets the desired $C_1$ blow-up pairs $(y_\pm,s_\pm)$ above and
below the genus and with $|y_\pm|\geq \gamma r_0$ where here
$\gamma$ is small and depends only on $C_1$. Lemma
\ref{BUPinConeLem} implies, after increasing $r_0$ if needed, that the
$y_\pm$ do not lie in $\mathbf{C}_{\delta}(0)$ and thus by
increasing $r_0$ further (by an amount depending only on $\gamma$,
$\delta$ and $h$) one has $x_3(y_+)>h>-h>x_3(y_-)$.
\end{proof}

Thus, we may iteratively construct the desired sequence of blow-up
pairs. This sequence will be used to construct the region
$\mathcal{R}_A$ in the next section.

\begin{prop} \label{AbvBelBUPprp} Given $\epsilon>0$ sufficiently small, there exist constants $C_1,
C_{in}>0$ and a sequence $(\tilde{y}_i,\tilde{s}_i)$ ($i \in
\mathbb{Z}\backslash \{0\}$) of $C_1$ blow-up pairs of $\Sigma$ such
that: the sheets associated to $(\tilde{y}_i,\tilde{s}_i)$ are
$\epsilon$-sheets on scale $\tilde{s}_i$ centered at $\tilde{y}_i$
and $x_3(\tilde{y}_i)<x_3(\tilde{y}_{i+1})$. Moreover, for $i\geq 1$,
$\tilde{y}_{i+1}\in B_{C_{in} \tilde{s}_i} (\tilde{y}_i)$ while for
$i\leq -1$, $\tilde{y}_{i-1}\in B_{C_{in}
\tilde{s}_i}(\tilde{y}_i)$.
\end{prop}
\begin{proof}
Without loss of generality, we work above the genus (i.e. for
$x_3>1$ and $i\geq 1$), as the argument below the genus is
identical. Let $\delta,R>0$ be given by Lemma \ref{BUPinConeLem}.
Thus, if $(y,s)$ is a blow-up pair in $\Sigma$ so that $|y|\geq R$
then $y\notin \mathbf{C}_{\delta}$.  Moreover, using $\epsilon$ and
$\delta$, let $\delta_0$ be given by Corollary \ref{OscCor} and
increase, if needed, $R$ as indicated by the corollary.  We are free
to shrink $\delta_0$, so assume that $\delta_0 \leq \epsilon$. Use
Theorem \ref{TwoInitShtExstThm} with  $\delta_0$ to choose $C_1,
C_2$ and increase $R$, if needed, as indicated by the theorem. Thus,
for any $(y,s)$ a $C_1$ blow-up pair with $|y|\geq R$, we have
$\delta_0$-sheets (which, as $\delta_0\leq\epsilon$ are also
$\epsilon$-sheets) associated to $(y,s)$. Moreover, this and the
choice of $R$ imply that Corollary \ref{OscCor} applies in
$\mathbf{C}_{\delta_0}(y)\backslash B_{2s}(y)$.

Corollary III.3.5 of \cite{CM3} and a standard blow-up argument
give constants $C_{out}>C_{in}>0$ such that, for a $C_1$ blow-up
pair $(y,s)$, as long as the component of $B_{C_{out} s} (y)\cap
\Sigma$ containing $y$ is a disk and there are blow-up sheets
associated to $(y,s)$, then we can find blow-up pairs
$(y_\pm,s_\pm)$ above and below $(y,s)$ (in a weak sense) and
inside $B_{C_{in} s} (y)$.
 If, in addition, we can apply Corollary
\ref{OscCor} centered at $y$,  then we can ensure $x_3(y_+)>
x_3(y)>x_3(y_-)$. Corollary \ref{ConditiononBUPsTHM} and
Proposition \ref{GlobTopProp} together give a value $h_1\geq R$,
depending on $C_{out}$, so for $|y|\geq h_1$ the component of
$B_{C_{out} s} (y)\cap \Sigma$ containing $y$ is a disk.

It now suffices to find an initial blow-up pair
$(\tilde{y}_1,\tilde{s}_1)$ with $x_3(\tilde{y}_1) \geq h_1$, as
repeated application of the argument of the above paragraph gives
the sequence $(\tilde{y}_i,\tilde{s}_i)$. Lemma \ref{firstbupprop},
with $h_1$ replacing $h$, gives the existence of the desired initial
blow-up pair.
\end{proof}
\section{Structural Decomposition of $\Sigma$}
We prove Theorem \ref{FirstMainThm} and Proposition
\ref{SecondMainThm} in subsection \ref{DecSec}.

\subsection{Constructing $\mathcal{R}_S$}

The decomposition of $\Sigma$ now proceeds as in Section 4 of
\cite{BB1}, with Proposition \ref{uThetaLowBndPrp} giving strict
spiraling far enough out in the $\epsilon$-sheets of $\Sigma$. After
specifying a region of strict spiraling, $\mathcal{R}_S$, the
remainder of $\Sigma$ will be split into the connected component
containing the genus, $\mathcal{R}_G$, and the region containing the
points of large curvature, $\mathcal{R}_A$.

In the interest of clarity we restate two results from \cite{BB1}
that we will need to prove our decomposition. The first result
gives the strict spiraling of $\epsilon$-sheets.
\begin{prop} (Proposition 3.3 in \cite{BB1})
\label{uThetaLowBndPrp} There exists an $\epsilon_0$ so: Suppose
 $\Gamma_u$ is a 3-valued $\epsilon$-sheet on scale 1 with
 $\epsilon<\epsilon_0$ and $w(1,\theta)=u(1, \theta + 2\pi)-u(1,\theta)\geq  C_2>0$.  Then
there exists $C_3=C_3(C_2)\geq 2$, so that on
$\pRect{C_3}{\infty}{-\pi}{\pi }$:
\begin{equation}
\label{uThetaLowBnd}
 u_\theta (\rho,\theta)\geq \frac{C_2}{8\pi}\rho^{-\epsilon}.
\end{equation}
\end{prop}
The second result is a technical lemma that will guarantee that any
sheets lying between sheets associated to consecutive blow-up pairs
are eventually (for large enough radius) $\epsilon$-sheets. Results
along these lines can by found in Section 5 of \cite{MMGPD} and
Section II.3 of \cite{CM3} . Importantly, the proof of such a
statement relies only on standard elliptic theory and the ability to
apply the one-sided curvature estimate in $\mathbf{C}_{\delta_1}(y)$
for an appropriately chosen $\delta_1$, where $y \in \Sigma$ is the
point of a blow-up pair. Lemma \ref{BUPinConeLem} and Corollary
\ref{OscCor} ensure, as long as we work far enough from the genus,
that this last condition is satisfied.  In order to avoid
technicalities, we restrict attention only to pairs
$(\tilde{y}_i,\tilde{s}_i)$ from Proposition \ref{AbvBelBUPprp}.
\begin{lem}
\label{thm211} There exists $\epsilon_0>0$ so: Given $N>4$ and
$\epsilon_0>\epsilon>0$ there exists $R_2=R_2(\epsilon,N)>1$ so that
if, using $\epsilon$, $(\tilde{y}_i,\tilde{s}_i)$ is a blow-up pair from Proposition
\ref{AbvBelBUPprp} with two associated $4$-valued $\epsilon$-sheets
$\Sigma_{j}$, $j=1,2$, then there exist two $N$-valued $\epsilon$-sheets
on scale $R_2\tilde{s}_i$, $\tilde{\Sigma}_j\subset \Sigma$.
Moreover, $\tilde{\Sigma}_j$ may be chosen so its $4$-valued
middle sheet contains $\Sigma_j\backslash \set{(x_1-x_1(\tilde{y}_i))^2+(x_2-x_2(\tilde{y}_i))^2\leq
R_2^2\tilde{s}_i^2}$.
\end{lem}
This is essentially Lemma 4.1 of \cite{BB1}, though the statement
there is technically simpler. As before, the only obstruction to
generalizing the proof from \cite{BB1} is the possibility that we
cannot apply the one-sided curvature estimates in
$\mathbf{C}_{\delta}(\tilde{y}_i)$ for some small $\delta$. However,
our choice of $\tilde{y}_i$ ensures this is not a problem.

We now wish to argue as in Lemma 4.3 of \cite{BB1} (where we
determine the regions $\mathcal{R}_A$ and $\mathcal{R}_S$ for
disks). To do so we must ensure that we may use the chord-arc bounds
of \cite{CY} near the pairs $(\tilde{y}_i,\tilde{s}_i)$.  By
choosing a subsequence of blow-up pairs $(y_i,s_i)$ that satisfy
this additional criteria, we obtain the following:

\begin{lem}\label{genusStructure}
There exist constants $C_1,R_0,R_1$ and a sequence $(y_i,s_i)$
($i\neq 0$) of $C_1$ blow-up pairs of $\Sigma$ so that:
$x_3(y_i)<x_3(y_{i+1})$ and for $i\geq 1$, $y_{i+1}\in B_{R_1 s_i}
(y_i)$ while for $i\leq -1$, $y_{i-1}\in B_{R_1 s_i}(y_i)$.
Moreover, setting $\tilde{\mathcal{R}}_A=\tilde{\mathcal{R}}_A^+\cup
\tilde{\mathcal{R}}_A^-$, where $\tilde{\mathcal{R}}_A^\pm$ is the
component of $\bigcup_{\pm i>0} \Sigma\cap B_{R_1s_i}(y_i)$
containing $y_{\pm 1}$, then $\Sigma \backslash
\left(\tilde{\mathcal{R}}_A\cup B_{R_0}\right)$ has exactly two
unbounded components $\Sigma^1$ and $\Sigma^2$, each of which are
strictly spiraling multi-valued graphs. We define the set
$\tilde{\mathcal{R}}_S=\Sigma^1\cup \Sigma^2$.
\end{lem}
\begin{proof}
Fix $\epsilon < \epsilon_0$ where $\epsilon_0$ is smaller than the
constants given by Proposition \ref{uThetaLowBndPrp} and Lemma
\ref{thm211}.  Using this $\epsilon$, let
$(\tilde{y}_i,\tilde{s}_i)$ be the sequence constructed in Lemma
\ref{AbvBelBUPprp}.  Let us now determine how to choose the sequence
$(y_i,s_i)$.

On $(y_i,s_i)$, we will need a uniform bound, $N$, on the number of
sheets between the blow-up sheets associated to the pairs $(y_i,
s_i)$ and $(y_{i+1},s_{i+1})$. This is equivalent to a uniform area
bound which in turn follows from the uniform curvature bounds of
Proposition \ref{genuscbprop} of the appendix, once we can establish
the appropriate chord-arc bounds. The proof of this is
straightforward and can be found in Lemma 4.2 of \cite{BB1}. Recall, from \cite{CY},
that the (strong) chord-arc bounds for minimal disks give a uniform constant
$\beta>1$ so for any $r$, if the component of $B_{2(r+1) \beta
s_i}(y_i)\cap\Sigma$ containing $y_i$ is a disk, then
 $B_{r s_i}(y_i)\cap \Sigma$ is a subset of
$\mathcal{B}_{(r+1) \beta s_i}(y_i)$. To correctly apply the argument of
Lemma 4.2 in \cite{BB1}, one must be sufficiently far from the genus;
i.e. for a fixed constant $C_{bnd}$, the component of
$B_{C_{bnd}s_i} (y_i)\cap \Sigma$ containing $y_i$ must be a disk.
Note that $C_{bnd}$ depends only on $\beta$ and $C_{in}$ (where
$C_{in}$ is as in Proposition \ref{AbvBelBUPprp}). To that end, pick
$h_2\geq 0$ by using Corollary \ref{ConditiononBUPsTHM} with
$\alpha^{-1} \geq \max\set{C_{bnd},2\beta (R_1+1)}$ where $R_1$ is to be
chosen later. We then pick the sequence $(y_i,s_i)$ from
$(\tilde{y}_i,\tilde{s}_i)$ by requiring $|x_3(y_i)|\geq h_2$ (and
then relabeling).  Notice that our method of choosing the $(y_i,s_i)$
ensures that $N$ is independent of our ultimate choice of $R_1$.

We now determine $R_1$ (see Figure \ref{geomlem}). By choice of
$(y_i,s_i)$, we can apply Lemma \ref{thm211}, so there exists an
$R_2$ such that all of the (at most) $N$ sheets between the blow-up
sheets associated to $(y_1,s_1)$ and $(y_2,s_2)$ are
$\epsilon$-sheets on scale $R_2 s_1$ centered on the line $\ell$
which goes through $y_1$ and is parallel to the $x_3$-axis.  Label
these pairs of $\epsilon$-sheets $\Sigma^k_j$, $k=1,2$ and $1 \leq j
\leq N$. Integrating \eqref{MainInEq21}, and using $N$ and $C_2$ we
get $\tilde{C}_2$ so $\tilde{C}_2 s_1$ is a lower bound on the
separation of each $\Sigma^k_j$ over the circle $\partial D_{R_2
s_1} (\Pi(y_{1}))\subset\set{x_3=0}$. Theorem \ref{uThetaLowBnd}
gives a $C_3$, depending on $\tilde{C}_2$, such that outside of a
cylinder centered at $\ell$ of radius $R_2C_3 s_1$, all the
$\Sigma_j^k$ strictly spiral. Choose $R_1$, depending only on
$C_{in}, N, \epsilon, C_3, \beta$ and $R_2$, so the component of
$B_{R_1 s_1}(y_1)\cap \Sigma$ containing $y_1$ also contains the
point $y_2$ and the intersection of this cylinder with each
$\Sigma^k_j$. This $R_1$ exists by the chord-arc bounds which we
have by the choice of $(y_i,s_i)$. As there was nothing special
about the blow-up pair $(y_1,s_1)$ in this argument and our
conclusions are invariant under a rescaling, we can apply the same
argument to each $(y_i,s_i)$ and thus construct
$\tilde{\mathcal{R}}_A$.
\begin{figure}
 \includegraphics[width=3in]{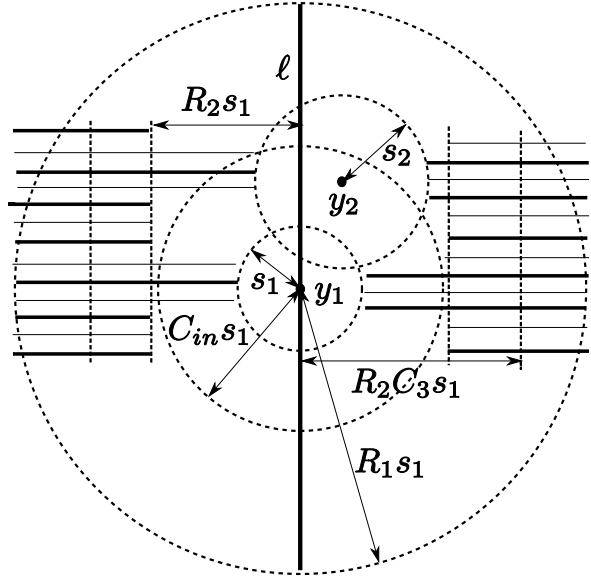}
\caption{An illustration of the proof of Lemma
\ref{genusStructure}.} \label{geomlem}
\end{figure}

Finally, by properness, there exists a finite number, $M$, of
$\epsilon$-sheets between the blow-up sheets associated to $(y_{\pm
1},s_{\pm 1})$. Pick $R_0$ large enough so that outside of the ball
of radius $R_0$ the $M$ sheets between the blow-up sheets associated
to $(y_1,s_1)$ and $(y_{-1},s_{-1})$ strictly spiral.  Such an $R_0$
exists by Proposition \ref{GraphNrConePrp}, Theorem \ref{thm211},
and the above argument.  By the above construction,  the $\Sigma^i$ are strictly spiraling
multi-valued graphs as described in Remark \ref{MultiGraphRem}.
\end{proof}

Notice that $\Sigma$ is not necessarily contained in
$\tilde{\mathcal{R}}_A \cup \tilde{\mathcal{R}}_S \cup B_{R_0}$. In
the next section, we will adjust these subsets in order to obtain
the decomposition.

\subsection{Decomposing $\Sigma$}\label{DecSec}The strict spiraling, the
fact that away from the genus convex sets meet $\Sigma$ in disks
(see Lemma \ref{GlobTopProp}) and the proof of Rado's theorem (see
\cite{OSS, RadoThm}) will give $\nabla_\Sigma x_3\neq 0$ in
$\mathcal{R}_A$. Then a Harnack inequality will allow us to bound
$|\nabla_\Sigma x_3|$ from below on $\mathcal{R}_A$.
We first use the strict spiraling on $\tilde{\mathcal{R}}_S$ and an
appropriate initial choice for $\mathcal{R}_G$ to determine the
behavior of the level sets of $x_3$.

\begin{proof}(Proposition
\ref{SecondMainThm}) By the properness of $\Sigma$ there exists an
$R_0'\geq R_0$, where $R_0$ is from Lemma \ref{genusStructure}, so
that the component of $B_{R_0'}\cap \Sigma$ containing
$\overline{\Sigma}$ contains $B_{R_0}\cap \Sigma$.  We take
$\mathcal{R}_G$ to be this component and define
$\Gamma=\Sigma\backslash \mathcal{R}_G$; note that $\partial
\mathcal{R}_G$ is connected by Proposition \ref{GlobTopProp}. By
increasing $R_0'$, if needed, we may assume that $\set{|x_3|\leq
2}\cap \partial \mathcal{R}_G \subset \tilde{\mathcal{R}}_S$. Notice
this implies that $\Gamma\cap \set{|x_3|\leq 2}\subset
\tilde{\mathcal{R}}_S=\Sigma^1 \cup \Sigma^2$ and so $\nabla x_3\neq
0$ in this set. Moreover, the strict spiraling on $\Sigma^1 \cup
\Sigma^2$ guarantees that $\Gamma\cap \set{x_3=c}$, for $|c| \leq
2$, consists of exactly two unbounded, smooth curves with boundary
on $\partial \Gamma$.

For $\set{|x_3|\geq 2}$ we now show that every level set $\{x_3 =
c\} \cap \Sigma$ consists of one smooth properly embedded curve.  We
use that $x_3$ is harmonic on $\Sigma$, the strict spiraling in
$\tilde{\mathcal{R}}_S$ and the proof of Rado's theorem.  The key
fact is that a non-constant harmonic function $h$ on a closed disk
has an interior critical point, $p$, if and only if the
 connected component of the level set $\set{h=h(p)}$ containing $p$ meets the boundary of the disk
in at least 4 points. For $|x_3|>1$, as the genus lies in $B_1$,
the intersection of $\Sigma$ with wide, short cylinders with axis
the $x_3$-axis are disks by the maximum principle and Proposition \ref{GlobTopProp}. Moreover, every
level set of $x_3$ can only have two ends by the strict spiraling.
The proof of Rado's theorem then immediately gives the non-vanishing
of the gradient for $|x_3|>1$ and so $\nabla_\Sigma x_3\neq 0$
in $|x_3|> 1$. In particular, $\set{x_3=c}\cap \Sigma$ is a
smooth curve for $|c|>1$.  The final statement of the proposition is then clear.
\end{proof}

In order to show Theorem \ref{FirstMainThm}, we need only
construct $\mathcal{R}_A$ and $\mathcal{R}_S$ from  the $\tilde{\mathcal{R}}_A$ and $\tilde{\mathcal{R}}_S$ of Lemma
\ref{genusStructure} and verify the lower bound on $|\nabla_\Sigma
x_3|$.
\begin{proof}(Theorem \ref{FirstMainThm}) We first verify that $|\nabla_\Sigma x_3|$ is bounded below on $\tilde{\mathcal{R}}_A$.
Suppose that $(y,s)$ is a blow-up pair in the sequence constructed
in Lemma \ref{genusStructure} and for convenience rescale so that
$s=1$. By our choice of blow-up pairs the constant $\beta$, we know
that every component of $B_{2\beta R_1}(y)\cap \Sigma$ is a disk
(where $R_1, \beta$ are from Lemma \ref{genusStructure}). Thus, the
component of $B_{R_1}(y) \cap \Sigma$ containing $y$ is contained in
$\mathcal{B}_{\beta {R}_1}(y) \subset B_{2\beta {R}_1}(y)\cap
\Sigma$.

Proposition \ref{genuscbprop} implies that curvature is bounded in
$B_{2 \beta {R}_1}(y) \cap \Sigma$ by some $K=K({R}_1)$. The
function $v=-2 \log |\nabla_\Sigma x_3| \geq 0$ is smooth by
Proposition \ref{SecondMainThm} and because, by construction,
$B_{2\beta {R}_1}(y) \cap \set{|x_3|\leq 1}=\emptyset$. Standard
computations give $\Delta_\Sigma v=|A|^2$. Then, since
$|\nabla_\Sigma x_3|=1$ somewhere in the component of $B_{
{R}_1} (y)\cap \Sigma$ containing $y$, we can apply a Harnack
inequality (see Theorems 9.20 and 9.22 in \cite{GiTr}) to obtain an
upper bound for $v$ on $\mathcal{B}_{\beta {R}_1}(y)$ that depends
only on
 $K$. Consequently,
there is a lower bound $\epsilon_1$ on $|\nabla_\Sigma x_3|$ in
the component of $\Sigma \cap B_{R_1}(y)$ containing $y$.  Since
this bound is scaling invariant, the same bound holds around any
blow-up pair from Lemma \ref{genusStructure}.

Recall, $\mathcal{R}_G$ is given by $\Sigma\backslash \Gamma$ where $\Gamma$ is from
Proposition \ref{SecondMainThm}.
Suppose $\Omega$ is a component of
$\Sigma\backslash ( \mathcal{R}_G\cup \tilde{\mathcal{R}}_A)$. By the construction of Lemma \ref{genusStructure}, $\Omega$ is either bounded or a subset of $\tilde{\mathcal{R}}_S$.  We need consider only bounded $\Omega$.  Notice
$\partial\Omega\subset
\partial (\mathcal{R}_G\cup \tilde{\mathcal{R}}_{A})\subset\partial \mathcal{R}_G \cup
\partial\tilde{\mathcal{R}}_A$. As $\partial\mathcal{R}_G$ is compact, and,
by construction, $\nabla_\Sigma x_3\neq 0$ on it, there exists some
 $\epsilon_2>0$ such that $|\nabla_\Sigma x_3| \geq \epsilon_2>0$ on
 $\partial \mathcal{R}_G$.  Let $\epsilon_0 =
 \min\{\epsilon_1,\epsilon_2\}$. Since $v$ is subharmonic, $|\nabla_\Sigma x_3|\geq \epsilon_0$ on
 $\Omega$.  Thus, define $\mathcal{R}_A$ to be the union of all these $\Omega$ with $\tilde{\mathcal{R}}_A\backslash \mathcal{R}_G$.  Set $\mathcal{R}_S=\Sigma\backslash \left(\mathcal{R}_A\cup \mathcal{R}_G\right)\subset \tilde{\mathcal{R}}_S$.
\end{proof}

\section{Conformal Structure of the End}
In Section \ref{pfsec} we prove Theorem \ref{ThirdMainThm} and
Corollary \ref{ThirdMainThmCor} by analysis similar to that in
Section 5 of \cite{BB1}.  We first show that
$\Gamma=\Sigma\backslash \mathcal{R}_G$ is conformally a punctured
disk and, indeed, the map $z=x_3+ix_3^*: \Gamma \to \mathbb{C}$ is a
proper, holomorphic coordinate. We then study the level sets. Recall
we let $x_3^*$ denote the harmonic conjugate of $x_3$. In order to show that
$z$ is a proper, holomorphic coordinate, one must check three
things: that $z$ is well defined, that it is injective and that it
is proper -- i.e. if $p\to \infty$ in $\Gamma$ then $z(p)\to
\infty$.  The first two statements are straightforward, whereas the
latter is far more subtle.
\begin{prop}
$z:\Gamma\to \mathbb{C}$ is a holomorphic coordinate.
\end{prop}
\begin{proof}
We first check $x_3^*$ is well defined on $\Gamma$.  As $\Sigma$ is minimal, ${}^*dx_3$, the conjugate differential to
$dx_3$, exists on $\Sigma$ and is closed and harmonic.  We wish to
show it is exact on $\Gamma$.  To do so, it suffices to show that
for every embedded closed curve $\nu$ in $\Gamma$, we have $\int_\nu {}^*dx_3
= 0$. By Proposition \ref{GlobTopProp}, $\Sigma \backslash \nu$ has
two components, only one of which is bounded. The bounded component,
together with $\nu$, is a manifold with (connected) boundary, and on
this manifold ${}^*dx_3$ is a closed form.  Hence, the result
follows immediately from Stokes' theorem.

We next check $z$ is injective on $\Gamma$.
First notice that, by Proposition \ref{SecondMainThm}, for any regular value $c$ of $x_3$, $\set{x_3=c}$ has exactly one unbounded curve and $x_3^*$ is strictly monotone along this curve.
Now suppose $p,q\in \Gamma$, $p\neq q$ and $x_3(p)=x_3(q)$ is a critical
value of $x_3$.  Note that $p$ and $q$ are regular points of $x_3$ --
as they lie on $\Gamma$ -- and so in a neighborhood of each point
$z$ is injective.  Clearly, there are points $p',q'\in \Gamma$
arbitrarily near $p,q$ so that $x_3(p')=x_3(q')=c'$ is a regular
value of $x_3$. Proposition \ref{SecondMainThm} implies the unbounded component, $\gamma$, of $\set{x_3=c'}\cap \Sigma$ contains $p'$ and $q'$. The fact that $z$ is injective near $p$ and the
monotonicity of $x_3^*$ on $\gamma$ together give  positive
lower bound on $|x_3^*(p')-x_3^*(q')|$ independent of $p',q'$. By continuity, this implies
a positive lower bound on $|x_3^*(p)-x_3^*(q)|$ and so $z(p)\neq z(q)$.
\end{proof}

\subsection{The winding number of the Gauss map}
In order to show that $z$ is proper we use the Gauss map of
$\Sigma$, or, more accurately, we use $g$, its stereographic
projection.  In particular, the logarithm of $g$, in $\Gamma$,
allows one to prove that $z$ is proper by complex analytic methods.
We will make this argument in Section \ref{zPropPrfSec}. However,
before we do so we must check such a logarithm is well-defined.
Notice as $\Gamma$ is an annulus it is not a priori clear that there
exists $f:\Gamma \to \mathbb{C}$ such that $g=e^f$ on $\Gamma$. For
such an $f$ to exist we must show that the (topological) winding
number of $g$ as a map from the annulus $\Gamma$ to the annulus
$\mathbb{S}^2\backslash \set{\pm (0,0,1)}$ is zero.  Because $g$ is
meromorphic in $\Sigma$ and has no poles or zeros in $\Gamma$, this
is equivalent to proving that
 $g$ has an equal number of poles and zeros.
\begin{prop}\label{exf}
Counting multiplicity, $g$ has an equal number of poles and zeros.
\end{prop}
\begin{proof}
The zeros and poles of $g$ occur only at the critical points of
$x_3$.  In particular, by Proposition \ref{SecondMainThm}, there
exist $h$ and $R$ so all the zeros and poles lie in the
cylinder:
  \begin{equation}
   C_{h,R}=\set{|x_3|\leq h, x_1^2+x_2^2\leq R^2}\cap \Sigma.
  \end{equation}
Moreover, for $R$ and $h$ large, $\gamma=\partial
C_{h,R}$ is the union of four smooth curves, two at the top and
bottom, $\gamma_t$ and $\gamma_b$, and two disjoint helix like
curves $\gamma_1,\gamma_2 \subset \mathcal{R}_S$.  Hence, for
$c\in (-h,h)$, $\set{x_3=c}$ meets $\partial C_{h,R}$ in exactly
two points.  Additionally, as $\gamma_1$ and $\gamma_2$ are
compact, there is a constant $\alpha>0$ so
$|\frac{d}{dt}x_3(\gamma_i(t))|>\alpha$, $i=1,2$.

Let us first suppose that $g$ has only simple zeros and poles and
these occur at distinct values of $x_3$; thus, the Weierstrass
representation implies that the critical points of $x_3$ are
non-degenerate. We now investigate the level sets $\set{x_3=c}$.
By the strict spiraling of $\gamma_i$ ($i=1,2$), at the regular
values these level sets consist of an interval with end points in
$\gamma_i$ ($i=1,2$) and the union of a finite number of closed
curves.  Moreover, by the minimality of $C_{h,R}$, the non-smooth
components of the level sets at critical values will consist of
either two closed curves meeting in a single point or the interval
and a closed curve meeting in a single point.  As a consequence of
this $\set{|x_3|\leq h, x_1^2+x_2^2\leq R^2}\backslash C_{h,R}$
has exactly two connected components $\Omega_1$ and $\Omega_2$.
Orient $C_{h,R}$ by demanding that the normal point into
$\Omega_1$. Notice it is well defined to say if a closed
curve appearing in $\set{x_3=c}\cap C_{h,R}$ surrounds $\Omega_1$
or $\Omega_2$.

\begin{figure}
 \includegraphics[width=3in]{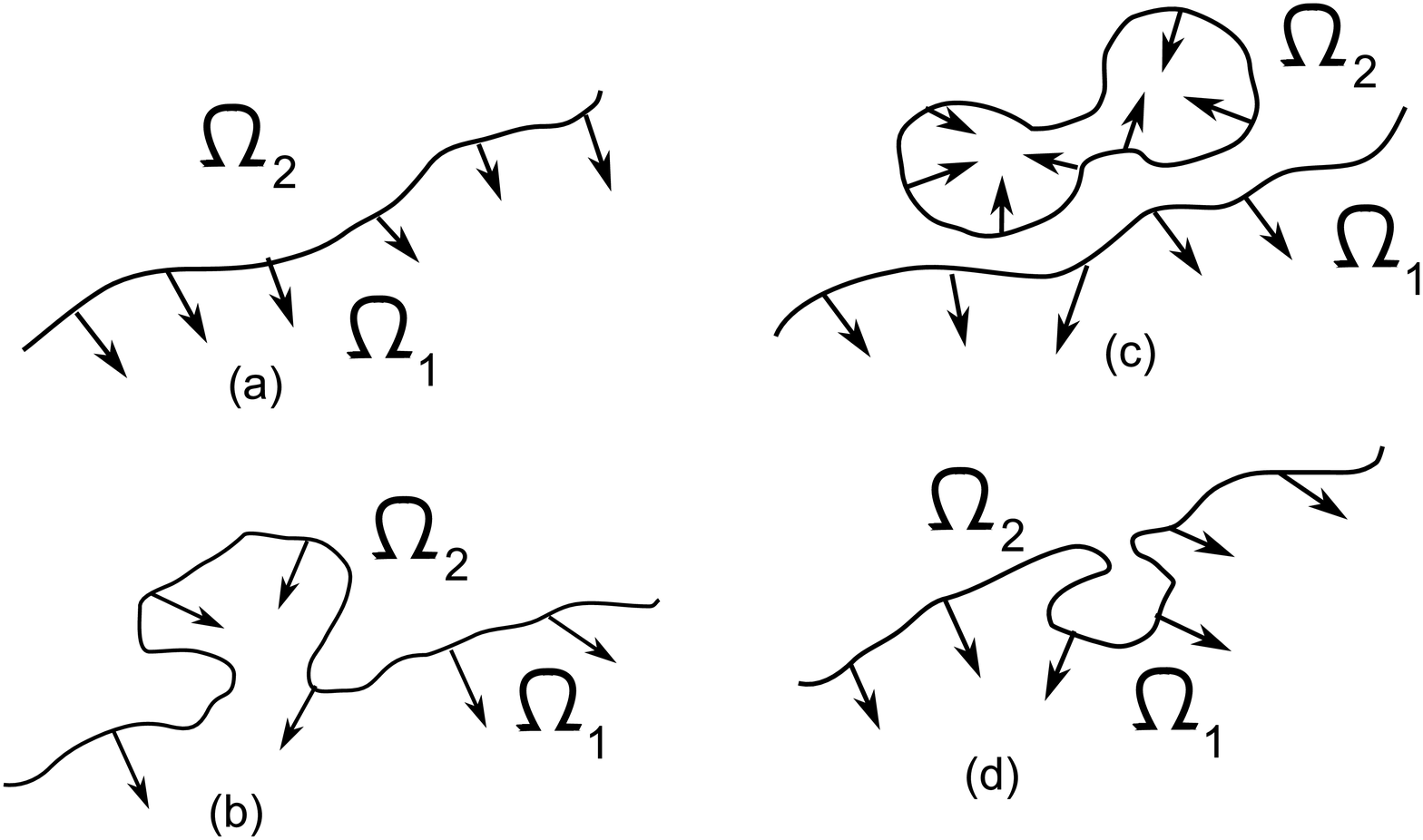}
\caption{Level curve examples in Proposition \ref{exf}. (a)
Initial orientation chosen at height $x_3=h$. (b) A curve pinching
off from $\Omega_1$. (c) Two curves pinching from one. (d) A curve
pinching off from $\Omega_2$.}\label{lvlcurves}
\end{figure}
The restrictions imposed on $g$ and minimality of $C_{h,R}$ imply
that at any critical level, as one goes downward, either a single
closed curve is ``created'' or is ``destroyed''. (See Figure
\ref{lvlcurves}.) Moreover, when such a curve is created it makes
sense to say whether it surrounds $\Omega_1$ or $\Omega_2$ and
this is preserved as one goes downward. Now suppose a closed curve
is created and that it surrounds $\Omega_1$; then it is not hard
to see that at the critical point the normal must point upwards.
Similarly, if a closed curve surrounding $\Omega_1$ is destroyed
then the normal at the critical point is downward pointing.  For
closed curves surrounding $\Omega_2$ the opposite is true; e.g.
when a closed curve is created, then at the critical point the
normal points downward. Thus, since the level sets at $h$ and $-h$
are intervals, one sees that the normal points up as much as it
points down. That is, $g$ has as many zeros as poles.

We now drop the restrictions on the poles and zeros of $g$. Beyond
these assumptions the argument above used only that $C_{h,R}$ was
minimal and that the boundary curves $\gamma_i$ ($i=1,2$) meet the level curves of $x_3$ in precisely one point.  It is not hard to check that these last two conditions are
preserved by small rotations around lines in the $x_1$-$x_2$ plane.
We claim that such rotations also ensure that the Gauss map of the
new surface must have simple poles or zeros and these are on
distinct level sets.  To that end we let $C_{h,R}^\epsilon$ be the
rotation of $C_{h,R}$ by $\epsilon$ degrees around a fixed line
$\ell$ in the $x_1$-$x_2$ plane and through the origin.

The strict spiraling of $\gamma_1,\gamma_2$ implies there exists an
$\epsilon_0>0$, depending on $\alpha$ and $R$ and a constant $K>0$,
depending on $R$, so: for all $0<\epsilon<\epsilon_0$, if $c\in
(-h+K\epsilon, h-K\epsilon)$ then $\set{x_3=c}\cap C_{h,R}^\epsilon$
meets $\partial C_{h,R}^\epsilon$ in two points. Moreover, by a
suitable choice of $\ell$ the critical points will be on distinct
level sets.  Denote by $g_\epsilon$ the stereographic projection of
the Gauss map of $C_{h,R}^\epsilon$. We now use the fact that $g$ is
meromorphic on $\Sigma$ (and thus the zeros and poles of $g$ are
isolated) and that $g_\epsilon$ is obtained from $g$ by a M\"obius
transform.  Indeed, these two facts imply that (after shrinking
$\epsilon_0$) for $\epsilon \in (0, \epsilon_0)$, $g_\epsilon$ has
only simple zeros and poles on $C_{h,R}^\epsilon$ and by our choice
of $\ell$ these are on distinct levels of $x_3$. To see this we note that there are $\alpha(\epsilon), \beta(\epsilon)\in \mathbb{C}$ (and also depending on $\ell$) satisfying $|\alpha(\epsilon)|^2+|\beta(\epsilon)|^2=1$ so,
\begin{equation}
g_\epsilon=\frac{\alpha(\epsilon) g-\bar{\beta}(\epsilon)}{\beta(\epsilon)g+\bar{\alpha}(\epsilon)},
\end{equation}
where we have also $|\alpha(\epsilon)|\neq 0,1$ and $\alpha(\epsilon)\to 1$ as $\epsilon\to 0$.  Thus, for $\epsilon$ sufficiently small all zeros of $g_\epsilon$ are distinct from, but near, zeros of $g$.  This implies that, at the zeros of $g_\epsilon$,  $dg_\epsilon$ does not vanish.

 By further
shrinking $\epsilon_0$ one can ensure that all of the critical
values occur in the range $(-h+K\epsilon, h-K\epsilon)$. Thus, the
level sets in $C_{h,R}^\epsilon$ of $x_3$ for $c\in (-h+K\epsilon,
h-K\epsilon)$ consist of an interval with endpoints in $\partial
C_{h,R}^\epsilon$, one in each $\gamma_i$ for $i=1,2$, and the
union of a finite number of closed curves. Our original argument
then immediately implies that $g_\epsilon$ has as many zeros as
poles. Notice this is equivalent to the vanishing of the winding number of the map $g_\epsilon$ restricted to
$\partial C_{h,R}^\epsilon$ (which is topologically
$\mathbb{S}^1$) as a map into $\mathbb{C}\backslash \set{0}$.
For $\epsilon$ sufficiently small, $g_\epsilon$ never
has a zero or pole on $\partial C_{h,R}^\epsilon$ and as long as
this is true, the winding number is independent of $\epsilon$.
Thus, $g$ has, counting multiplicity,  the same number of poles and
zeros.
\end{proof}

\begin{cor}
A holomorphic function $f:\Gamma\to \mathbb{C}$ exists so $e^f=g$
on $\Gamma$.
\end{cor}

\subsection{The conformal structure of the end}\label{zPropPrfSec}
The strict spiraling in $\mathcal{R}_S$ is used in \cite{BB1} to
show that the logarithm of $g$, i.e. $f=f_1+if_2$, is, away from a
neighborhood of $\mathcal{R}_A$, a proper conformal diffeomorphism
onto the union of two disjoint closed half-spaces. Since every
level set of $x_3$ has an end in each of these sets, the
properness of $z$ was then a consequence of Schwarz reflection and
the Liouville theorem.  The
 proof only used properties of the end of the surface and so holds also in $\mathcal{E}(1,+)$:
\begin{prop}
\label{hpoly} There exists a $\gamma_0>0$ so:  $f$ is a proper
conformal diffeomorphism from $\Omega_\pm$ onto $\set{z:\pm\re z\geq
2\gamma_0}\subset \mathbb{C}$, where \begin{equation}\label{omegadef}
\Omega_\pm=\set{x\in \Gamma: \pm f_1(x)\geq 2\gamma_0}\subset
\Gamma.\end{equation}
\end{prop}

Proposition 5.1 in \cite{BB1} asserts and proves the identical
statement for minimal disks. The proof relies on showing there
exists $\gamma_0$ such that for every regular value $\gamma \geq
2\gamma_0$, $f_1^{-1}(\gamma)$ consists of exactly one curve on
$\Sigma$, which lies on every sheet of one of the components of
$\mathcal{R}_S$ (note in \cite{BB1} $\log g$ is denoted by $h$).  We
rely on the fact that $|\nabla_\Sigma x_3|$
is a function of $|f_1|$.
Recall, $\nabla_\Sigma x_3$ is the projection of $\mathbf{e}_3=\nabla_{\Real^3} x_3$
onto $T\Sigma$,  and so $|\nabla_\Sigma x_3|$ can be expressed
terms of the $x_3$-coordinate of the unit normal to $T\Sigma$.  Thus, by computing the inverse stereographic projection, one obtains:
\begin{equation}
\label{gradx3h} |\nabla_\Sigma x_3| =2 \frac{ |g|}{1+|g|^2}\leq
2e^{-|f_1|}.
\end{equation} By Theorem \ref{FirstMainThm}, as $|\nabla_\Sigma x_3| \geq \epsilon_0 >0$ on $\mathcal{R}_A \cup \partial \mathcal{R}_G$, there exists
$\gamma_0>0$ so on $\mathcal{R}_A \cup \partial \mathcal{R}_G$,
$|f_1(z)|\leq \gamma_0$. The proof in \cite{BB1} only requires that
$f_1^{-1}(\gamma)$ lies in $\mathcal{R}_S$; thus, since
$f_1^{-1}(2\gamma_0) \cap
\partial \Gamma = \emptyset$, using $f_1^{-1}( \gamma)\cap \Gamma$, the proof carries over without change. The interested reader should consult Proposition 5.1 in
\cite{BB1} for the details.

\subsection{The proofs of Theorem \ref{ThirdMainThm} and Corollary
\ref{ThirdMainThmCor}}\label{pfsec}

In Proposition 5.2 of \cite{BB1}, we show that for $\Sigma \in
\mathcal{E}(1,0)$, $f \circ z^{-1}: \mathbb{C} \to \mathbb{C}$ is
linear. The result follows from standard complex analysis,
exploiting both Schwarz reflection and Liouville's Theorem. For
$\Sigma\in \mathcal{E}(1,+)$, there are a few necessary, but simple,
modifications.
\begin{proof}(Theorem \ref{ThirdMainThm}) We first show that $x_3^*\to \pm \infty$
along each level set of $x_3$; that is $z:\Gamma \to \mathbb{C}$ is
a proper holomorphic coordinate. This follows easily once we
establish that each level set of $x_3$ has one end in $\Omega_+$ and
the other in $\Omega_-$, where these sets are as defined in
\eqref{omegadef}. This follows from the radial gradient decay on
level sets of $x_3$ forced by the one-sided curvature estimate.
Indeed, Corollary \ref{OscCor} and the structural decomposition of
$\Sigma$ imply that for any $\epsilon>0$ small, there is a point
$y_\epsilon \in \Sigma$ and a $\delta_\epsilon>0$ so that within
suitable subsets of $\mathbf{C}_{\delta_\epsilon} (y_\epsilon)$,
$\Sigma$ must be a graph with gradient bounded by $\epsilon$. Recall
\eqref{GradDecay} says that, for any $\delta_\epsilon$-sheet, there
is sub-linear gradient decay on the sheet and so it must eventually
lie within $\mathbf{C}_{\delta_\epsilon} (y_\epsilon)$. Thus, by
Corollary \ref{OscCor} for some large $R_\epsilon>0$ every point of
$\Sigma \cap \mathbf{C}_{\delta_\epsilon} (y_\epsilon)\backslash
B_{R_\epsilon}(y_\epsilon)$ lies on some multi-valued graph that has
gradient bounded by $\epsilon$. Notice any level set of $x_3$ has
its ends in this set and so $|\nabla_\Sigma x_3|\leq C \epsilon$ in
a neighborhood of the ends.

Thus, $x_3(\partial \Omega_+)=(-\infty,\infty)$ and so $z(\partial \Omega_+)$ splits
$\mathbb{C}$ into two components with only one, $V$, meeting
$z(\Omega_+)=U$. If $U$ is a proper subset of $V$ then, by conformally straightening the boundary of $V$ and precomposing with
$f^{-1}|\Omega_+$, we can apply Schwarz reflection to get a map
from $\mathbb{C}$ into a proper subset of $\mathbb{C}$. The
Liouville Theorem then implies that $z$ is constant. This gives a
contradiction and so $U=V$, i.e. $x_3^*\to \pm \infty$ along each level set of $x_3$.  Thus, $z(\Gamma)$ contains $\mathbb{C}$ with a
closed disk removed; in particular, $\Gamma$ is conformally a
punctured disk.  Since $f_1^{-1}(\gamma_0) \cap \Gamma$ is a single smooth curve, $f$
has a simple pole at the puncture. Similarly, by Proposition
\ref{SecondMainThm}, $z$ has a simple pole at the puncture.  In
$\Gamma$, the height differential $dh=dz$ and $\frac{dg}{g}=df$, proving the theorem.
\end{proof}
Embeddedness and the Weierstrass representation,
\eqref{WeierstrassRep}, then give Corollary
\ref{ThirdMainThmCor}:
\begin{proof}(Corollary \ref{ThirdMainThmCor})
Theorem \ref{ThirdMainThm} gives that, in $\Gamma$, $f(p)=\alpha
z(p)+\beta +F(p)$ where $\alpha,\beta\in \mathbb{C}$ and
$F:\Gamma\to \mathbb{C}$ is holomorphic and has holomorphic
extension to the puncture (and has a zero there).  By translating
$\Sigma$ parallel to the $x_3$-axis and re-basing $x_3^*$ we may
assume $\beta=0$. By Proposition \ref{SecondMainThm},
$\set{x_3=0}\cap \Gamma \subset \mathcal{R}_S$ can be written as the
union of two smooth proper curves, $\sigma_\pm$, each with one end
in $\partial \Gamma$, and parametrized so $x_3^*(\sigma_\pm(t))=t$
for $\pm t>T$.

Note that, since $\sigma_\pm'(t)$ is perpendicular to both $\mathbf{e}_3$ and to the normal $\mathbf{n}$ to $\Sigma$ at $\sigma_\pm(t)$, the projection of $\mathbf{n}$ onto
the $\{x_3=0\}$ plane is also perpendicular to $\sigma_\pm'(t)$.
This projection is, up to the correct identification of $\set{x_3=0}$ with $\mathbb{C}$, parallel to $g(\sigma_\pm(t))$.  Since
$\arg(g(\sigma_\pm(t)))= (\re \alpha) t + \im F(\sigma^\pm(t))$, we
see that $\arg(\sigma_\pm'(t)) = \pm \pi/2 + (\re \alpha) t + o(1/t)$.
Consider, for a moment, the curve $\sigma_+(t)$. If $\re \alpha \neq 0$, $\arg(\sigma_+'(t)) \to \infty$ as $t \to \infty$. Thus,
$\sigma_+$ hits the $x_1$-axis infinitely many times.  As $\Sigma$ is properly embedded, this set of intersections tends to $\infty$.  Note that the same can be said for $\sigma_-$, but the
choice of parametrization means it spirals in the opposite
direction.  Thus, the two curves must intersect which contradicts
embeddedness. Therefore, $\re \alpha=0$.
\end{proof}
\appendix
\section{Topological structure of $\Sigma$}
\label{TopStrApp} An elementary but crucial consequence of the
maximum principle is that each component of the intersection of a
minimal disk with a closed ball is a disk.  Similarly, each
component of the intersection of a genus $k$ surface with a ball
has genus at most $k$ (see Appendix C of \cite{CM4} and Section I
of \cite{CM3}). We note that for $\Sigma$ with one end and finite
genus we obtain a bit more:
\begin{prop} \label{GlobTopProp}
Suppose $\Sigma\in \mathcal{E}(1)$ and $\overline{\Sigma} \subset
\Sigma\cap B_1$ is smooth and connected, with the same genus as $\Sigma$. Then,
$\Sigma\backslash \overline{\Sigma}$ is an annulus. Moreover, for
any convex set $C$ with non-empty interior, if $C\cap
B_1=\emptyset,$ then each component of $C\cap \Sigma$ is a disk.
Alternatively, if $B_1\subset C$ then all components of
$C\cap\Sigma$ not containing $\overline{\Sigma}$ are disks.
\end{prop}
\begin{proof}
That $\Sigma'=\Sigma\backslash \overline{\Sigma}$ is an annulus is a purely
topological consequence of $\Sigma$ having one end. Indeed, as the Euler characteristic satisfies
$\chi(\Sigma)=\chi(\Sigma')+\chi(\overline{\Sigma})$ and $\Sigma$ has one end, one computes that
$2g(\Sigma')+e(\Sigma')+e(\overline{\Sigma})=3$, where $g(X)$
and $e(X)$ respectively represent the genus and number of punctures of
$X$. On the other hand, as $\Sigma$ has one end,  $e(\Sigma')=1+e(\overline{\Sigma})$ proving the claim.

If $C$ and $B_1$ are disjoint then, as they are convex, there exists
a plane $P$ so that $P$ meets $\Sigma$ transversely and so that $P$
separates $B_1$ and $C$. Since $\Sigma \backslash \overline{\Sigma}$
is an annulus and $P\cap \overline{\Sigma}=\emptyset$, the convex
hull property implies that $P\cap \Sigma$ consists only of unbounded
smooth proper curves. Thus, exactly one of the components of
$\Sigma\backslash P$ is not a disk. As $C$ is disjoint
from this component we have the desired result.

If $C$ is convex and contains $B_1$, denote by $\Gamma$ the component of
$C\cap \Sigma$ containing $\overline{\Sigma}$.  Suppose $\Gamma_1$ is a different component of $\Sigma \cap C$.  Let $\gamma$ be a component of $\partial \Gamma_1$.  As $\Sigma\backslash \overline{\Sigma}$ is an annulus, we have that $\gamma$ is separating, in particular, one component, $\Gamma_2$ of $\Sigma\backslash \gamma$ is pre-compact. Clearly, $\Gamma_2$ either contains $\overline{\Sigma}$ or is a disk.  By the convex hull property, one has that $\Gamma_2\subset C$ and so $\Gamma_2$ is a component of $\Sigma \cap C$. Thus, by the strong maximum principle $\Gamma_1=\Gamma_2$ and so $\Gamma_1$ is a disk.
\end{proof}

\section{Geometry near a blow-up pair} \label{BUPGeomSec}
The existence of a blow-up pair $(y,s)$ in a minimal surface
$\Sigma$, by definition, implies uniform bounds on the
geometry in the ball $B_s(y)$.  Colding and
Minicozzi's work shows further that there are uniform bounds on
the geometry in any ball on the scale of $s$.  This is most
easily proved using their lamination results. Indeed, we have the
following uniform bound on the curvature, which is an extension of
Lemma 2.26 of \cite{CY} to surfaces of finite genus:
\begin{prop}\label{genuscbprop}
Given  $K_1,g$ we get a constant $K_2$ such that if
\begin{enumerate}
\item \label{genuscbpropa}$\Sigma \subset \Real^3$ is an embedded
minimal surface with genus$(\Sigma)=g$ \item $\Sigma \subset
B_{K_2s}(y)$ and $\partial \Sigma \subset
\partial B_{K_2s}(y)$
\item  \label{genuscbpropc}$(y,s)$ is a blow-up pair,
\end{enumerate}
then we get the curvature bound
\begin{equation}\label{genuscb}
\sup_{B_{K_1s}(y)\cap\Sigma}|A|^2 \leq K_2s^{-2}.
\end{equation}
\end{prop}
The proof is nearly identical to that of Lemma 2.26 of \cite{CY}.
That proof is by contradiction, using Colding and Minicozzi's
compactness result for minimal disks, i.e. Theorem 0.1 of
\cite{CM4}. One proves Proposition \ref{genuscbprop} by
the same argument, but uses instead a more general compactness result, i.e.
Theorem 0.6 of \cite{CM5}.
\section{One-sided Curvature in $\Sigma$}\label{oscsec} In several places we make
use of the one-sided curvature estimate of \cite{CM4}.  Recall that
this result gives a curvature estimate for a minimal disk that is
close to and on one side of a plane. As a sequence of rescaled
catenoids shows, it is crucial that the surface be a disk, something
that makes application to surfaces in $\mathcal{E}(1,+)$ somewhat
subtle. Nevertheless, Proposition \ref{GlobTopProp} allows the use
of the one-sided curvature estimate far from the genus. Recall the
statement of the estimate:
\begin{thm}\label{oscthm}(Theorem 0.2 of \cite{CM4})
There exists $\epsilon>0$ so that if $\Sigma \subset B_{2r_0} \cap
\{x_3>0\} \subset \Real^3$ is an embedded minimal disk with
$\partial \Sigma \subset \partial B_{2r_0}$, then for all
components, $\Sigma'$ of $\Sigma \cap B_{r_0}$ which intersect
$B_{\epsilon r_0}$ we have
\begin{equation}
\sup_{\Sigma'} |A_{\Sigma}|^2 \leq r_0^{-2}.
\end{equation}
\end{thm}

A particularly important consequence of
 Theorem \ref{oscthm} is Corollary I.1.9 of \cite{CM4}, which roughly states that
if an embedded minimal disk has a two-valued graph in the
complement of a cone (and away from a ball), then all components
of $\Sigma$ in the complement of a larger cone (and larger ball)
are multi-valued graphs. Essentially, the two-valued graph takes
the place of the plane in Theorem \ref{oscthm}.  This fact was used extensively in \cite{BB1}.  Thus, we give two variants of it that hold for elements of $\mathcal{E}(1,+)$ and indicate how they follow from \cite{CM4}.

\begin{cor}\label{OscCor0}
 There exists a $c>1$ so: Given $\epsilon>0$ there exist $\delta_0>0$ so that if $\mathbf{C}_{c\delta_0}(y)\backslash B_{s}$ does not meet $\overline{\Sigma}$ and contains a weak 2-valued $\delta_0$-sheet centered at $y$ and on scale $s$, then each component of $\Sigma \cap (\mathbf{C}_{\delta_0}(y)\backslash
B_{2s}(y))$ is a multi-valued graphs with gradient bounded by
$\epsilon$.
\end{cor}
\begin{proof}
The result follows immediately from the proof of Corollary I.1.9
of \cite{CM4} (as long as one notes that the proof of Corollary
I.1.9 depends only on each component of $\Sigma\cap
(\mathbf{C}_{c\delta_0}(y)\backslash B_{s(y)})$ meeting any convex
set in a disk (and so Theorem \ref{oscthm} applies) for $c$
a universal constant. Proposition \ref{GlobTopProp} and the hypothesis ensure this.
\end{proof}
We will also use the following specialization of the above:
\begin{cor}\label{OscCor} Given
$\epsilon,\delta>0$ there exist $\delta_0 > 0$ and $R>1$ such
that, if there exists a weak 2-valued $\delta_0$-sheet centered at
$y$ on scale $s$  where $y\notin \mathbf{C}_{\delta}\cup B_R$,
then all the components of $\Sigma\cap
(\mathbf{C}_{\delta_0}(y)\backslash B_{2s}(y))$ are multi-valued
graphs with gradient $\leq \epsilon$.
\end{cor}
\begin{proof}
The result follows immediately from Corollary \ref{OscCor0} as
long as we can ensure that $\overline{\Sigma}\subset B_1(0)$, is
disjoint from $\mathbf{C}_{c\delta_0}(y) \backslash B_{s}(y)$.
Suppose $x\in \mathbf{C}_{c\delta_0}(y)$ and think of $x$ and $y$
as vectors.  By choosing $\delta_0$ sufficiently small, depending
on $\delta$, we have that $|\langle x-y,y\rangle|
<(1-\gamma)|y||x-y|$ (that is the angle between $x-y$ and $y$ is
bounded away from $0^\circ$); note $1>\gamma>0$ depends only on
$\delta$.  But then $|x|^2=|x-y+y|^2\geq |x-y|^2+2\langle
x-y,y\rangle +|y|^2\geq \gamma|y|^2$.  Hence, picking
$R^2>\frac{1}{\gamma}$ suffices.
\end{proof}


\section{Colding-Minicozzi Lamination Theory} \label{GlobGeomApp} \label{LamResApp} We note that Theorem
\ref{FirstMainThm} is a sharpening, for $\Sigma\in \mathcal{E}(1)$,
of a much more general description of the shapes of minimal surfaces
given by Colding and Minicozzi in \cite{CM5}. More precisely, in
that paper they show, for a large class of embedded minimal surfaces
in $\Real^3$, how the geometric structure of a surface is determined
by its topological properties.  In particular, as $\Sigma$ has
finite topology and one end, their work shows that it roughly looks
like a helicoid. That is, away from a compact set containing the
genus, $\Sigma$ is made up of two infinite-valued graphs that spiral
together and are glued along an axis.  Using this description, they
show compactness results that generalize their lamination theory of
\cite{CM4}.  As in the case for disks, the derivation of this global
description of finite genus surfaces uses local versions of
propositions as in Section \ref{StrSec}. However, Colding and
Minicozzi do not explicitly state these results and so, for the sake
of completeness, we will state a modified form of a compactness
result from \cite{CM5} and use it to give simple proofs of
Propositions \ref{GraphNrConePrp} and \ref{GraphExtPrp}.

While the lamination theory of \cite{CM5} will be the launching
point for proving the two propositions, we need only outline one
small portion of the theory to get our result. In particular, we
need only consider the structure of the limit lamination of
homothetic dilations for $\Sigma \in \mathcal{E}(1)$. In this case,
the lamination has the same structure as for a sequence of embedded
minimal disks, which is modeled on rescalings of the helicoid.

\begin{thm}
Let $\Sigma \in \mathcal{E}(1)$ be non-flat, and let $\lambda_i \to
0$. Set $\Sigma_i = \lambda_i \Sigma$. There exists a subsequence
$\Sigma_j$, a foliation $\mathcal{L}=\{x_3=t\}_{t\in \Real}$ of
$\Real^3$ by parallel planes, and a closed nonempty set
$\mathcal{S}$ in the union of the leaves of $\mathcal{L}$ such that
after a rotation of $\Real^3$:
\begin{enumerate}
\item For each $1>\alpha>0$, $\Sigma_j \backslash \mathcal{S}$
converges in the $C^\alpha$-topology to the foliation
$\mathcal{L}\backslash \mathcal{S}$.
\item $\sup_{B_r(x)\cap \Sigma_j}|A|^2 \to \infty$ as $j \to \infty$
for all $r>0$ and $x \in \mathcal{S}$.  (The curvatures blow up
along $\mathcal{S}$.)
\item \label{twosing}Away from $\mathcal{S}$, each $\Sigma_j$
consists of exactly two multi-valued graphs spiraling together.
\item $\mathcal{S}$ is a single line orthogonal to the
leaves of the foliation.
\end{enumerate}
\end{thm}
\begin{rem}For the theorem in its entirety, see Theorem 0.9 of
\cite{CM5}.
\end{rem}
We now use the nature of this convergence to deduce gradient bounds
outside a cone. This, together with further application of the
compactness theorem, gives Propositions \ref{GraphNrConePrp} and
\ref{GraphExtPrp}.
\begin{lem} \label{GraphBndCor}
For any $\epsilon>0, \delta>0$ there exists an $R>1$ so every
component of $(\mathbf{C}_{\delta}\backslash B_R)\cap \Sigma$ is a
graph over $\set{x_3=0}$ with gradient less than $\epsilon$.
\end{lem}
\begin{proof}
We proceed by contradiction.  Suppose there exists a sequence
$\{R_i\}$ with $R_i \to \infty$ and points $p_i\in
(\mathbf{C}_{\delta}\backslash B_{R_i})\cap \Sigma$ such that the
component of $B_{\gamma|p_i|} (p_i)\cap \Sigma$ containing $p_i$,
$\Omega_i$, is not a graph over $\set{x_3=0}$ with gradient less
than $\epsilon$. Here $\gamma$ depends on $\delta$ and will be
specified later.  Now, consider the sequence of rescalings
$\frac{1}{|p_i|}\Sigma$, which by possibly passing to a subsequence
converges to $\mathcal{L}$ (away from $\mathcal{S}$). Passing to
another subsequence, $\frac{1}{|p_i|} p_i$ converges to a point
$p_\infty \in \mathbf{C}_{\delta}\cap B_1$.  Let $\tilde{\Omega}_i =
\frac{1}{|p_i|} \Omega_i$.  Proposition III.0.2 of \cite{CM5}
guarantees that if $B_{\gamma} (p_\infty) \cap \mathcal{S}=
\emptyset$ then the $\tilde{\Omega}_i$ converge to
$\tilde{\Omega}_\infty \subset \{x_3=x_3(p_\infty)\}$ as graphs. As
$\mathcal{S}$ is the sole singular set, we may choose $\gamma$
small, depending only on $\delta$, to ensure this. Thus, for large
$j$, $ \tilde{\Omega}_j$ is a graph over $\set{x_3=0}$ with gradient
bounded by $\epsilon$, giving the desired contradiction.
\end{proof}
We now show Propositions \ref{GraphNrConePrp} and \ref{GraphExtPrp}:
\begin{proof} (of Proposition \ref{GraphNrConePrp}). Let $\tilde{\delta} = \epsilon$ and $\tilde{\epsilon} \leq
\tilde{\delta}/(4\pi N)$. Choose $R$ from Lemma \ref{GraphBndCor},
using this $\tilde{\delta}, \tilde{\epsilon}$. Thus, every component
of $\Sigma \cap \mathbf{C}_{\tilde{\delta}} \backslash B_R$ has
gradient bounded by $\tilde{\epsilon}$. Since $|w(\rho, \theta)|
\leq \int_0^{2\pi} |u_\theta| \leq 2\pi \tilde{\epsilon}\rho$, we
see that there are $N$ sheets in $\mathbf{C}_\epsilon$.
\end{proof}
\begin{proof} (of Proposition \ref{GraphExtPrp}) Note that as long as $|y|$ is sufficiently large, Theorem 0.6 of
\cite{CM2} gives an $\Omega<1/2$ (as well as a constant $C_1$) so
that since the component of $B_{\frac{1}{2}|y|}(y) \cap \Sigma$
containing $y$ is a disk, there exists a $N$-valued graph $\Sigma_0$
over the annulus, $A= D_{\Omega|y|}\backslash D_{s/2}(y)\subset P$
with gradient bounded by $\epsilon/2$. Here $P$ is in principle an
arbitrary plane in $\Real^3$.

We claim that Lemma \ref{GraphBndCor} implies a subset, $\Sigma_0'$,
of $\Sigma_0$ is a $N$-valued graph over the annulus $A'=D_{\Omega
|y|/2}\backslash D_{s}(\Pi(y))\subset \set{x_3=0}$ with gradient
bounded by $\epsilon$, which further implies $\Sigma_0'$ can be
extended as desired. To that end we note that for $\delta >
1/(4\Omega)$, if $y\notin \mathbf{C}_{\delta}$ then $A$ (and thus,
by possibly increasing $\delta$, $\Sigma_0$) meets
$\mathbf{C}_{\delta}$. Lemma \ref{GraphBndCor} allows us to choose
an $R_0>0$ so that every component of $\Sigma \cap
(\mathbf{C}_{\delta}\backslash B_{R_0})$ is a multi-valued graph
over $\set{x_3=0}$ with gradient bounded by $\epsilon/4$.  Thus if
we take $R>2R_0$ then there is a point of $\Sigma_0$ in
$\mathbf{C}_{\delta}\backslash B_{R_0}$; therefore, for the gradient
estimates at the point to be consistent, $P$ must be close enough to
$\set{x_3=0}$ so that we may choose $\Sigma_0'\subset\Sigma_0$ so it
is a multi-valued graph over $A'$.  Furthermore, the part of
$\Sigma_0'$ over the outer boundary of $A'$ is necessarily inside of
$\mathbf{C}_{\delta}\backslash B_{R_0}$ and so Lemma
\ref{GraphBndCor} allows us to extend it as desired.
\end{proof}

\bibliographystyle{amsplain}
\bibliography{ConfStrucFinal5}

\providecommand{\bysame}{\leavevmode\hbox to3em{\hrulefill}\thinspace}
\providecommand{\MR}{\relax\ifhmode\unskip\space\fi MR }
\providecommand{\MRhref}[2]{%
  \href{http://www.ams.org/mathscinet-getitem?mr=#1}{#2}
}
\providecommand{\href}[2]{#2}
\begin{thebibliography}{10}

\bibitem{BB1}
J.~Bernstein and C.~Breiner, \emph{Helicoid-like minimal disks and uniqueness},
  Preprint. {\tt http://arxiv.org/abs/0802.1497}.

\bibitem{CM5}
T.~H. Colding and W.~P.~Minicozzi II, \emph{{The space of embedded minimal
  surfaces of fixed genus in a 3-manifold V; Fixed genus}}, Preprint.

\bibitem{MMGPD}
\bysame, \emph{Multivalued minimal graphs and properness of disks}, Int. Math.
  Res. Not. (2002), no.~21, 1111--1127.

\bibitem{EXC}
\bysame, \emph{An excursion into geometric analysis}, Surv. Differ. Geom.
  \textbf{IX} (2004), 83--146.

\bibitem{CM1}
\bysame, \emph{{The space of embedded minimal surfaces of fixed genus in a
  3-manifold I; Estimates off the axis for disks}}, Ann. of Math. (2)
  \textbf{160} (2004), no.~1, 27--68.

\bibitem{CM2}
\bysame, \emph{{The space of embedded minimal surfaces of fixed genus in a
  3-manifold II; Multi-valued graphs in disks}}, Ann. of Math. (2) \textbf{160}
  (2004), no.~1, 69--92.

\bibitem{CM3}
\bysame, \emph{{The space of embedded minimal surfaces of fixed genus in a
  3-manifold III; Planar domains}}, Ann. of Math. (2) \textbf{160} (2004),
  no.~2, 523--572.

\bibitem{CM4}
\bysame, \emph{{The space of embedded minimal surfaces of fixed genus in a
  3-manifold IV; Locally simply connected}}, Ann. of Math. (2) \textbf{160}
  (2004), no.~2, 573--615.

\bibitem{CY}
\bysame, \emph{{The Calabi-Yau conjectures for embedded surfaces}}, Ann. of
  Math. (2) \textbf{167} (2008), no.~1, 211--243.

\bibitem{Co}
P.~Collin, \emph{Topologie et courboure des surfaces minimales proprement
  plongees de $\mathbb{R}^3$}, Ann. of Math. (2) \textbf{145} (1997), 1--31.

\bibitem{GiTr}
D.~Gilbarg and N.~S. Trudinger, \emph{Elliptic partial differential equations
  of second order}, Springer-Verlag, 1998.

\bibitem{HPR}
L.~Hauswirth, J.~Perez, and P.~Romon, \emph{{Embedded minimal ends of finite
  type}}, Trans. Amer. Math. Soc. \textbf{353} (2001), no.~4, 1335--1370.

\bibitem{HoffmanMeeksHalfSpace}
D.~Hoffman and W.~H.~Meeks III, \emph{{The strong halfspace theorem for minimal
  surfaces}}, Invent. Math. \textbf{101} (1990), no.~1, 373--377.

\bibitem{HKW}
D.~Hoffman, H.~Karcher, and F.~Wei, \emph{Global analysis in modern
  mathematics}, ch.~The Genus One Helicoid and the Minimal Surfaces that Led to
  its Discovery, Publish or Perish, 1993.

\bibitem{HoffmanMcCuan}
D.~Hoffman and J.~McCuan, \emph{{Embedded minimal ends asymptotic to the
  Helicoid}}, Commun. Anal. Geom. \textbf{11} (2003), no.~4, 721--736.

\bibitem{WHW2}
D.~Hoffman, M.~Weber, and M.~Wolf, \emph{The genus-one helicoid as a limit of
  screw-motion invariant helicoids with handles}, Clay Math. Proc., 2, Global
  theory of minimal surfaces, pp.~243--258.

\bibitem{WHW}
\bysame, \emph{An embedded genus-one helicoid}, Ann. of Math. (2) \textbf{169}
  (2009), no.~2, 347--448.

\bibitem{HWe}
D.~Hoffman and F.~Wei, \emph{{Deforming the singly periodic genus-one
  helicoid}}, Experiment. Math. \textbf{11} (2002), no.~2, 207--218.

\bibitem{HW2}
D.~Hoffman and B.~White, \emph{The geometry of genus-one helicoids}, Comment.
  Math. Helv., To Appear.

\bibitem{HW}
\bysame, \emph{Genus-one helicoids from a variational point of view}, Comment.
  Math. Helv. \textbf{83} (2008), no.~4, 767--813.

\bibitem{Huber}
A.~Huber, \emph{{On subharmonic functions and differential geometry in the
  large}}, Comment. Math. Helv. \textbf{32} (1958), no.~1, 13--72.

\bibitem{MR2}
W.~H.~Meeks III and H.~Rosenberg, \emph{{The geometry and conformal structure
  of properly embedded minimal surfaces of finite topology in $\Real^3$}},
  Invent. Math. \textbf{114} (1993), no.~3, 625--639.

\bibitem{MR}
\bysame, \emph{The uniqueness of the helicoid}, Ann. of Math. (2) \textbf{161}
  (2005), no.~2, 727--758.

\bibitem{OSS}
R.~Osserman, \emph{A survey of minimal surfaces}, Dover Publications, New York,
  1986.

\bibitem{RadoThm}
T.~Rado, \emph{{On the problem of Plateau}}, Ergebnisse der Math. und ihrer
  Grenzgebiete \textbf{2} (1953).

\bibitem{WHWPNAS}
M.~Weber, D.~Hoffman, and M.~Wolf, \emph{An embedded genus-one helicoid}, Proc.
  Nat. Acad. Sci. \textbf{102} (2005), no.~46, 16566--16568.

\end{thebibliography}

\end{document}